\documentclass[letterpaper, 10 pt, conference]{ieeeconf}  % Comment this line out
\IEEEoverridecommandlockouts                              % This command is only
\title{\LARGE \bf
Ergodic control of a heterogeneous population and application to electricity pricing
}

\author{Quentin Jacquet, Wim van Ackooij, Clémence Alasseur and Stéphane Gaubert% <-this % stops a space
\thanks{Q. Jacquet, W. van Ackooij and C. Alasseur are with EDF R\&D Saclay, Palaiseau, France
        {\tt\small \{quentin.jacquet, wim.van-ackooij, clemence.alasseur\}@edf.fr}}%
\thanks{Q. Jacquet and S. Gaubert are with INRIA, CMAP, Ecole Polytechnique, CNRS, Palaiseau, France
        {\tt\small stephane.gaubert@inria.fr}}%
}

\usepackage{comment}
\let\labelindent\relax
\usepackage{enumitem}
\usepackage{dsfont,amsmath,amssymb,mathtools}

\DeclareMathOperator*{\st}{s.t.}

\DeclareMathOperator*{\Relint}{relint}
\DeclareMathOperator*{\Conv}{vex}
\DeclareMathOperator{\Lip}{Lip}

\DeclareMathOperator*{\Span}{sp}
\DeclareMathOperator{\Diam}{Diam}
\DeclareMathOperator{\OneF}{1}
\DeclareMathOperator*{\vex}{Vex}
\DeclareMathOperator*{\argmax}{arg\,max}
\DeclareMathOperator*{\argmin}{arg\,min}

\DeclareMathOperator{\bbR}{\mathbb{R}}
\DeclareMathOperator{\bbN}{\mathbb{N}}

\DeclareMathOperator{\calX}{\mathcal{X}}

\DeclareMathOperator{\calS}{\mathcal{S}}
\DeclareMathOperator{\calA}{\mathcal{A}}

\DeclareMathOperator{\calD}{\mathcal{D}}
\DeclareMathOperator{\BelOp}{\mathcal{B}}

\newcommand{\R}{\mathbb{R}}

\newcommand\hh{\hat{{h}}}  
\newcommand\hg{\hat{{g}}}  
\DeclareMathOperator{\Prob}{\mathbb{P}}

\DeclareMathOperator{\Conti}{\mathcal{C}^0}
\newcommand{\shortminus}{\scalebox{0.5}[1.0]{$-$}}

\usepackage{hyperref}
\usepackage{footnote}

\usepackage{multirow,hhline,colortbl,booktabs}
\usepackage{dsfont}

\DeclareMathSymbol{\mlq}{\mathord}{operators}{``}
\DeclareMathSymbol{\mrq}{\mathord}{operators}{`'}

\makeatletter
\newcommand*\frob{\mathpalette\bigcdot@{.7}}
\newcommand*\bigcdot@[2]{\mathbin{\vcenter{\hbox{\scalebox{#2}{$\m@th#1\bullet$}}}}}

\usepackage[%disable,
  colorinlistoftodos,bordercolor=orange,backgroundcolor=orange!20,linecolor=orange,textsize=scriptsize]{todonotes}

\usepackage[font=small,justification=centering]{caption}
\usepackage[justification=centering]{subcaption}

\usepackage{textcomp}
\usepackage{algorithm,algpseudocode}
\algdef{SE}[DOWHILE]{Do}{doWhile}{\algorithmicdo}[1]{\algorithmicwhile\ #1}%

\algnewcommand{\IfThenElse}[3]{% \IfThenElse{<if>}{<then>}{<else>}
  \State \algorithmicif\ #1\ \algorithmicthen\ #2\ \algorithmicelse\ #3}

\newcounter{savealgorithm}

\def\namedlabel#1#2{\begingroup
    #2%
    \def\@currentlabel{#2}%
    \phantomsection\label{#1}\endgroup
}

\usepackage[backend=biber,
style=alphabetic,
firstinits,
maxbibnames=9,
sorting=nyt,
url=false]{biblatex}
\usepackage[capitalise]{cleveref}

\newtheorem{prop}{Proposition}%[section]
\newtheorem{theorem}{Theorem}%[section]
\newtheorem{defin}{Definition}%[section]
\newtheorem{corol}{Corollary}%[section]
\newtheorem{lemma}{Lemma}%[section]
\newtheorem{remark}{Remark}%[section]
\newtheorem{example}{Example}%[section]
\crefname{lemma}{lemma}{lemmas}
\Crefname{lemma}{Lemma}{Lemmas}
\crefname{theorem}{theorem}{theorems}
\Crefname{theorem}{Theorem}{Theorems}
\crefname{prop}{proposition}{propositions}
\Crefname{prop}{Proposition}{Propositions}
\crefname{defin}{definition}{definitions}
\Crefname{defin}{Definition}{Definitions}
\crefname{example}{example}{examples}
\Crefname{example}{Example}{Examples}
\crefname{remark}{remark}{remarks}
\Crefname{remark}{Remark}{Remarks}

\usepackage[backend=biber,
style=alphabetic,
firstinits,
maxbibnames=9,
sorting=nyt,
url=false]{biblatex}
\begin{document}

\maketitle
\thispagestyle{empty}
\pagestyle{empty}

%%%%%%%%%%%%%%%%%%%%%%%%%%%%%%%%%%%%%%%%%%%%%%%%%%%%%%%%%%%%%%%%%%%%%%%%%%%%%%%%

\begin{abstract}
We consider a control problem for a heterogeneous population 
composed of agents able to switch at any time between different
options.
The controller aims to maximize an average gain per time unit, supposing that the population is of infinite size.
This leads to an ergodic control problem for a ``mean-field" Markov Decision Process in which the state space is a product of simplices, and the population evolves
according to controlled linear dynamics. By exploiting contraction properties of the dynamics in Hilbert's projective metric, we prove that the infinite-dimensional ergodic eigenproblem admits a solution and show that the latter is in general non unique.
This allows us to obtain optimal strategies, and to quantify the gap between steady-state strategies and optimal ones.
In particular, we prove in the one-dimensional case that there exist cyclic policies  -- alternating between discount and profit taking stages -- which secure a greater gain than constant-price policies. On numerical aspects, we develop a policy iteration algorithm with ``on-the-fly" generated transitions, specifically adapted to decomposable models, leading to substantial memory savings. We finally apply our results on realistic instances coming from an electricity pricing problem encountered in the retail markets, and numerically observe the emergence of cyclic promotions for sufficient inertia in the customer behavior.
\end{abstract}

\section{Introduction}

\subsection{Ergodic control for mean-field MDPs}
Many control problems involve a large number of rational agents, reacting to the decisions of a controller such as in finance~\cite{Bielecki_1999,Bauerle_2011}, routing problems~\cite{Calderone_2017} or epidemiology~\cite{Lee_2021} .
To overcome the intractability that appears when the number of individuals grows, mean-field control have been introduced, see e.g.~\cite{Bensoussan_2013}. 
Assuming that the agents in the population are indistinguishable, the fundamental idea is to apply a mean-field type approximation and to show that looking at the population distribution (instead of each individual state) is sufficient. 
In particular, early convergence results (of order $1\sqrt{N}$) for the $N$-cooperative agents control problem to the mean-field limit have been proved by Gast and Gaujal~\cite{Gast_2011} for discounted horizons. 
Motte and Pham~\cite{Motte_2022} generalize the latter results in the presence of common-noise. We focus here on the \emph{ergodic control problem} (i.e., infinite undiscounted horizon and average long-term rewards). In this context, it is showed in~\cite{Bauerle_2023} that the optimal policy in the mean-field limit is $\varepsilon$-optimal for the $N$-agents discounted version when the size of the population is large and the discount factor is close to one.

The ergodic control problem for a Markov decision
process with Bellman operator $\mathcal{B}$, on a compact
state space $\mathcal{D}$,
is classically studied by means of the ergodic eigenproblem
\begin{align}
g 1_{\mathcal{D}} + h = \mathcal{B}h \enspace,
\end{align}
in which $h$ is a bounded function on the state space,
called the bias or potential, and $g$ is a real constant. We refer to~\cite{Hernandez_Lerma_1996} for background on the topic.
If the ergodic eigenproblem is solvable, then, $g$ yields
the optimal mean payoff per time unit, and it is independent of the initial state. Moreover, an optimal
policy can be obtained by selecting maximizing actions
in the expression of $\mathcal{B}h$.
When the state and action spaces are finite, the ergodic eigenproblem
is well understood, in particular, a solution
does exist if every policy yields a unichain transition
matrix (i.e., a matrix with a unique final class), see e.g.~\cite{Puterman_1994}.
%\todo{SG: reference to be added, and check whether finite actions are needed or not}

Here, the ergodic control problem arises from a mean-field approximation (called \emph{lifted} MDP in~\cite{Motte_2022}). The state space is therefore infinite and corresponds to the space of probability distributions over the choices proposed to a representative agent, typically for a finite number of choices $\{1,\hdots,n\}$, $\mathcal{D}$ is the probability simplex $\Delta_n = \{\mu\in\bbR^n_+: \sum_{i=1}^n \mu_i = 1\}$. In this context of infinite state space, the solution to the ergodic eigenproblem is a more difficult
question~\cite{KM,fathi,Nuss-Mallet,agn,Akian_2016}. 
This is particularly the case in absence of common-noise -- where the lifted MDP is of {\em deterministic} nature -- owing to the lack of regularizing
effect coming from stochasticity.

\subsection{Contributions}
We consider a population of agents, that have different types/preferences. Each agent chooses between several options, taking into account the actions made by the controller, who aims at optimizing a mean reward per time unit.
This is represented by a discrete-time ergodic control problem, in which the state --the population-- belongs to a product of simplices. We suppose that the population evolves according to the Fokker-Planck equation of a controlled
Markov chain. In this work, we directly study the ``mean-field'' model where the population is supposed to be of infinite size. This choice is motivated by our application on the French electricity market where the population is in fact the whole set of French households (around $30$ millions), leading to intractable model without such mean-field hypothesis.
Our first main result, Theorem~\ref{prop::ergodic_fixed_point}, shows that the ergodic eigenproblem does admit a solution (in the presence of common noise or not). This entails that the value of the ergodic control problem
is independent of the initial state, and this also allows us to determine optimal stationary strategies. Theorem~\ref{prop::ergodic_fixed_point} requires a primitivity assumption on the semigroup of transition matrices; it applies in particular to positive transition matrices, such as the ones arising from logit based models. The proof relies on contraction properties of the dynamics in Hilbert's projective metric, which allow us to establish compactness estimates which guarantee the existence of a solution. 

In order to numerically solve the ergodic eigenproblem, we develop a policy iteration algorithm building on~\cite{Cochet_1998,Denardo_1968} but especially intended for \emph{decomposable} state spaces (e.g., for populations of independent customers) where transitions can be generated on-the-fly using pre-computed local information, see~\Cref{algo::Howard}. This refined version is afterwards compared with existing approaches, see~\Cref{sec::results}, and reveals drastic computational time reductions with respect to the value-iteration algorithm and considerable memory gains in comparison with off-the-shelf policy iteration procedures.  

In addition, we study stationary pricing strategies. Owing to the contraction properties of the dynamics, these are such that the population distribution converges to a stationary state.
%% constant control ( a particular focus on steady-state strategies (policies that converges to a specific case).
Then, we refine a result from~\cite{Flynn_1979}, providing a bound
on the loss of optimality arising from the restriction to stationary pricing
strategy. We define a family of Lagrangian functions, whose duality gap provides an explicit bound on the optimality loss, see~\Cref{prop::dual_lagrangian}. In particular, a zero duality gap guarantees that stationary pricing policies are optimal.

Finally, we apply our results to a problem of electricity pricing,
inspired by a real case study (French contracts). An essential
feature of this model is to take into account the {\em inertia} of customers, i.e., their tendency to keep their current contract even if it is not the best offer. This is represented by a logit-based stochastic transition model with switching costs. Theorem~\ref{prop::proba_stat} provides a closed-form formula for the stationary distribution, which allows us to determine steady-state policies by reduction to a single-level problem. We also obtain qualitative results through majorization concepts~\cite{Marshall_2011}, showing that the addition of inertia in the model leads to a more concentrated distribution of the population, see~\Cref{prop::majorization}.
We present numerical tests on examples of dimension $2$ and $4$. These reveal
the emergence of optimal cyclic policies for large switching costs, recovering the empirical notion of ``promotions" of~\cite{Dube_2009} and~\cite{Pavlidis_2017}.
%% =======
%% Finally, we study an explicit model of electricity pricing. This models takes
%% into account the {\em inertia} of customers, i.e., their tendency to keep
%% their current contract even if it is not the best offer.
%% Thus is represented by a logit-based stochastic transition model with switching costs. \Cref{prop::proba_stat} provides a closed-form formula for the stationary probability. %\Cref{todo} establishes a qualitative feature of this model: if the price is kept constant, then the stationary distribution majorizes the one obtained in the corresponding logit-model without inertia. In other words, incorporated inertia in the model leads to a more concentrated distribution of population.
%% We present numerical tests on examples of dimension $2$ and $4$, which reveal a typical behavior: the emergence of cyclic policies for large switching costs, recovering the empirical notion of ``promotions" of~\cite{Dube_2009} and~\cite{Pavlidis_2017}.
%% >>>>>>> 0c23c1d677a1cc7eb6d2c0e4d27428e98857a895

\subsection{Related works}
As mentioned above, we allow here the ergodic eigenproblem to be of deterministic nature, more degenerate than its stochastic analogue studied in the context of average cost Markov Decision Processes, see e.g.~\cite{Arapostathis_1993} and the references therein. The main classical approach to show the existence of a solution to the infinite-dimensional ergodic eigenproblem relies on a Doeblin-type (or minorization) condition. The latter entails a contraction property of a Markov semi-group acting on spaces of measures, as well as the contraction of the Bellman operator with respect to the Hilbert pseudo-norm. It implies not only the existence but also the \emph{uniqueness} of the ergodic eigenvector (up to a constant). This method has been used in~\cite{Kurano_1989,Hernandez_Lerma_1996}, and more recently in the works of Biswas~\cite{Biswas_2015} and Wiecek~\cite{Wiecek_2019} in the mean-field context. We also refer the reader to~\cite{Bansaye_2019,Gaubert_2014} for background on Doeblin (and the more general Dobrushin) type conditions. In our setting, this approach does not apply.
In fact, we provide an explicit counter example  showing that the eigenvector may not be unique, see~\Cref{ex::non_uniqueness}, and this entails that our existence result cannot be obtained using a Doeblin-type approach. 

Instead, we exploit here the contraction properties of
the dynamics to show that the family of value functions of the associated discounted problem is equi-Lipschitz. Then, following a now classical approach of Lions, Papanicolaou and Varadan~\cite{Lions_1987}, a solution of the ergodic eigenproblem is obtained by a vanishing discount limit. The use of contraction ideas is partly inspired by a previous work of Calvez, Gabriel
 and Gaubert~\cite{1404.1868}, tackling a different problem (growth maximization).
Also, \cite{1404.1868} deals with a PDE rather than discrete setting.
Our result may also be compared with those of B\"auerle~\cite{Bauerle_2023}, in which the equi-Lipschitz property of the value function is supposed a priori.

In the deterministic setting, the ergodic eigenproblem is a special case of the
``max-plus'' or ``tropical'' infinite dimensional spectral problem~\cite{KM,Akian_2009},
or of the eigenproblem studied in discrete weak-KAM and Aubry Mather theory~\cite{fathi,thieullen}. 
Basic spectral theory results require the Bellman operator to
be compact. This holds under demanding ``controllability''
conditions (see e.g.~\cite[Theorem 3.6]{KM}), not satisfied in our setting. Extensions of these results rely on quasi-compactness techniques~\cite{Nuss-Mallet,agn}. 

We also note that the ergodic eigenproblem, in the special deterministic ``$0$-player case'',
has been studied under the name of cohomological equation
in the field of dynamical systems. The existence of a regular
solution is generally a difficult question, a series of results
going back to the work of Liv\v{s}ic~\cite{Livsic_1972}, show that a H\"older continuous eigenvector does exists
if the payment function is H\"older continuous, and if the dynamics
is given by an Anosov diffeomorphism. The latter condition
requires the tangent bundle of the state space to split in two
components, on which the dynamics is either uniformly expanding or uniformly
contracting. Here, we establish a ``one player'' version, but requiring
a uniform contraction assumption.

On the computational aspect, a standard approach to solve the ergodic eigenproblem is to use the relative value-iteration (RVI) algorithm which goes back to White~\cite{White_1963}. Its convergence requires a demanding primitivity condition, see~\cite{Federgruen_1978}. It has been remarked in~\cite{Gaubert_2020} that this can be relaxed by combining RVI with Krasnoselskii-Mann damping. In the worstcase, $\varepsilon^{-2}$ iterations are needed to solve the problem with a precision $\varepsilon$, see e.g.~(\Cref{algo::RVI}). Here, in the present mean-field case, an essential step is to discretize the state space, wich is a product of simplices. Recently, in the model-free context and for infinite discounted horizons, RVI algorithms have been used by Carmona, Laurière and Tan~\cite[Algo.~1]{Carmona_2021}, also considering a beforehand discretization. A different class of algorithm rely on policy-iteration (PI), still relying on a discretization. In the deterministic case, PI can be implemented by a fast graph algorithm, see e.g.~\cite{Cochet_1998}. Here, we exploit the decomposability of the transition matrix (the dynamics of the populations are only coupled by the control) to obtain a more scalable method. Different refinements of policy iteration have been developed by Festa~\cite{Festa_2018} using domain decomposition to obtain parallel Howard's algorithm, as well as memory space gain. Bayraktar, B\"auerle and Kara~\cite{Bayraktar_2023} prove convergence results of the solution of the discretize version to the continuous one for discounted payoff by exhibiting regret bounds between the approximated mean-field MDP obtained by a semi-Lagrangian discretization (nearest neighbor)  and the true infinite-population case. To this purpose, they fully exploit the Lipschitz property of the optimal value function~\cite[Lemma 6]{Bayraktar_2023} by supposing that the dynamics is contracting for the Wasserstein metric. 

Finally, on the application side, we aim at analyzing the impact of switching costs on the retail electricity market, and especially on the optimal behavior of the retailers. In Economics, consumers are often supposed to be fully rational, and their reactions to price to be instantaneous. 
However, many studies highlight that switching costs and limited awareness conjointly lead to inertia in retail electricity markets, which hinders efficient choices, see~\cite{Hortaccsu_2017,Ndebele_2019,Dressler_2021}. 
Inertia in imperfect markets impacts the decision of the providers and modifies their pricing strategies. Studies tend to show the importance of promotions in the pricing behaviors of firms, see~\cite{Horsky_2010,Allender_2012}. In particular, empirical analyses show how the depth and frequency of promotions are linked with the level of inertia. Here, we study the problem through a mathematical angle using mean-field MDPs. In comparison with the work of Pavlidis and Ellickson~\cite{Pavlidis_2017} on multiproduct pricing, we also consider logit-based transitions but we focus on long-term average rewards and reinforce the theoretical understanding of the model by identifying the optimal steady-states and studying the emergence of cyclic policies.

This article is organized as follows.
In Section~\ref{sec::fixed_point}, we first define the model and prove the results on the ergodic eigenproblem (existence and non-uniqueness). We then present two iterative algorithms to solve this fixed-point problem in Section~\ref{sec::resolution}. We study steady-states and their optimality in Section~\ref{sec::lagrangian}, and illustrate the electricity application in Section~\ref{sec::case_study}. 
%{\em The proofs of the main results are given in the appendix}. 

\textit{An initial account of some of the present results appeared in the conference paper~\cite{Jacquet_2022}.}
%% (by comparison, \Cref{} is now in a more general setting, allowing non-deterministic infinite dimensional Markov decision processes, counter example showing the non-uniqueness of the bias function has been added, 

%%%%%%%%%%%%%%%%%%%%%%%%%%%%%%%%%%%%%%%%%%%%%%%%%%%%%%%%%%%%%%%%%%%%%%%%%%%%%%%%
\section{Ergodic control}\label{sec::fixed_point}
% -----
\subsection{Notation}
We denote by $\left<\cdot,\cdot\right>_n$ the scalar product on $\bbR^n$, and for any $x$ and $y$ in $\bbR^n$, $x\vee y$ represents the elementwise maximum between $x$ and $y$. We recall that the probability simplex of $\bbR^n$ is denoted by $\Delta_n = \{\mu\in\bbR^n_+: \sum_{i=1}^n \mu_i = 1\}$. Besides, we denote by
$\Span(f):= \max_{x\in E}f(x) - \min_{x\in E}f(x)$ the span of the function $f:E\to \R$.
%For a function $f:E\to F$, we denote by
We say that a matrix $P$ is positive, and we write $P\gg 0$, if all the coefficients of $P$ are positive. The set of convex functions with finite real values on a space $K$ is denoted by $\vex{K}$, and the convex hull of a set $K$ is denoted by $\Conv{K}$. Moreover, the set of Lipschitz function on $E$ is denoted by $\Lip(E)$, and the relative interior of a set $E$ is denoted by $\Relint(E)$.

The {\em Hilbert projective metric} $d_H$ on $\bbR^n_{> 0}$ is defined as $d_H(u,v) = \max_{1\leq i,j \leq n}\log(\frac{u_i}{v_i}\frac{v_j}{u_j})$, see~\cite{Lemmens_2009}. It is such that $d_H(u,v)=0$ iff the vectors $u$ and $v$ are proportional, hence, the name ``projective''.
For a set $E\subseteq \bbR^n_{> 0}$, we denote by $\Diam_H(E):=\max_{u,v \in E}d_H(u,v)$ the diameter of the set $E$, and for a matrix $P \in \bbR^{n\times n}$ we denote by $\Diam_H (P) := \max_{1\leq i,j\leq n} d_{H}(P_i,P_j)$ the {\em diameter} of $P$, where $P_i$ denotes the $i$th row of $P$. This can be seen to coincide with the diameter, in Hilbert's projective metric,
of the image of the set $\R_{>0}^n$ by the matrix $P$,
see Th.~A.6.2, {\em ibid.}%]{Lemmens_2009}.%\todo{SG: add ref?}

Finally, for a sequence $(a_t)_{t\geq 1}$, we respectively denote by $a_{s:t}$ and $a_{:t}$ the subsequences $(a_\tau)_{s\leq \tau \leq t}$ and $(a_\tau)_{1\leq \tau \leq t}$. 

% -----
\subsection{Model}
We consider a large population model composed of $K$ clusters of indistinguishable individuals. 
Each cluster $k\in[K]:=\{1,\dots,K\}$ represents a proportion $\rho_k$ of the overall population, and is supposed to react \emph{independently} from the other clusters.

Let $\calX$ and $\calA$ be respectively the state and action spaces. We suppose in the sequel that $\calX$ is finite and w.l.o.g. $\calX=\{1,2,\dots,N\}$. We suppose also that $\calA$ is a compact set (in~\Cref{sec::case_study}, $\calA$ will be an explicit subset of $\bbR^N$).

For any time $t\geq 0$ and any cluster $k$, we denote by $x^k_t \in [N]$ the stochastic choice made by a representative agent of cluster $k$ at time $t$. The distribution of the population of cluster $k$ over $[N]$ is then denoted by $\mu^k_t = \left(\Prob\left[x^k_t = i\right]\right)_{i\in[N]}\in \Delta_N$. We suppose that the dynamics of the process $\mu^k$ is given by a discrete time (linear) equation of the form
\begin{equation}\label{eq::linear_transition}
\mu^k_t = \mu^k_{t-1}P^k(a_t,\xi_t)\enspace,
\end{equation}
where $P^k$ is the Markov transition matrix of the underlying process $x^k$. 

In the first instance, we consider that the latter matrix is impacted by an exogenous process (\emph{common-noise}), independent of the initial state $\mu_0$, and represented by a sequence of independent and identically distributed (i.i.d.)
random variables $\{\xi_t\}$ with values in some space $\Xi$ and common distribution $\sigma$.

At every time $t\geq 1$, a controller chooses an action $a_t\in \calA$. She obtains a stochastic reward $r:\calA \times \Delta_N^K \times \Xi \to \bbR$ defined as 
\begin{equation}\label{eq::shape_r}
r:(a_t,\mu_t,\xi_t)\mapsto\sum_{k\in[K]}  \rho_k \left<\theta^{k}(a_t,\xi_t),\mu^{k}_t\right>_N\enspace,
\end{equation} 
where $\theta^{k}(a,\xi)\in\bbR^N$ is the vector whose entry $n$ represents the unitary reward for the controller coming from an individual of cluster $k$ in state $n$, for realization $\xi$ and after executing action $a$.

The semi-flow $\phi$ describing the dynamics of the state $\mu$ is then defined by a function depending on the past actions and past realizations of the common-noise:
$$\phi_t(a_{:t},\xi_{:t},\mu_0):= \mu_t\enspace.$$ 
%We also denote by $d:\Delta_N^K\to\calA$ a decision rule, i.e., $a = d(\mu)$ is the action taken by the controller if the population is distributed according to $\mu$. 
We also denote by $\Pi$ the set of policies. Then, for a given policy $\pi=\{\pi_t\}_{t\geq 1}$, the action taken by the controller at time $t$ is $a_t = \pi_t(\mu_t)$.

In the sequel, the following assumptions will be used:
\begin{description}%[leftmargin=0.35in]
\item[\namedlabel{hypo::continuity_P}{(A1)}] The transition $(a,\xi)\mapsto P^k(a,\xi)$ is a continuous function for any $k$.
\item[\namedlabel{hypo::primitivity_P}{(A2)}] There exists $L\in \bbN$ such that for any sequence of actions $a_{:L}\in\calA^L$, any sequence $\xi_{:L}$ and cluster $k$, $\prod_{l\in[L]} P^k(a_i,\xi_i) \gg 0$.
\end{description}
Recall that in Perron-Frobenius theory, a nonnegative matrix $M$
is said to be {\em primitive} if there is an index $l$ such that $M^l\gg 0$,
see~\cite[Ch.~2]{Berman_1994}. Assumption \ref{hypo::primitivity_P} holds in particular under
the following elementary condition:
\begin{description}%[leftmargin=0.39in]
\item[\namedlabel{hypo::positiveness_P}{(A2')}] For any  $a\in\calA$, cluster $k$ and $\xi\in\Xi$, $P^k(a,\xi) \gg 0$.
\end{description}
We will also make the following assumption:
\begin{description}
\item[\namedlabel{hypo::bounded_r}{(A3)}]
There exists a constant $M_r$ such that, $|\theta^{kn}(a,\xi)|\leq M_r$ for every $k\in[K]$, $n\in[N]$, $a\in\calA$ and $\xi\in\Xi$.
\end{description}
Condition~\ref{hypo::primitivity_P} has appeared in~\cite{Gaubert_1996} in the context
of semigroup theory, it can be checked algorithmically
by reduction to a problem of decision for finite semigroups,
see Rk.~3.8, {\em ibid.} 
Observe that~\ref{hypo::bounded_r} is very reasonable in practice.

We equip the product of simplices $\Delta_N^K$ with the norm
$\|\mu\|:= \sum_{k=1}^K \|\mu^k\|_{1}$. It follows from~\ref{hypo::bounded_r} that for any action $a$ and realization $\xi\in\Xi$, the total reward function $\mu\mapsto r(a,\mu,\xi)$ is a $M_r$-Lipschitz real-valued function from $(\Delta_N^K,\|\cdot\|)$ to $(\bbR,|\cdot|)$.

%\begin{remark}
%Assumption~\ref{hypo::positiveness_P} is not a Doeblin minorization condition (as in~\cite{Biswas_2015} for instance) : this would suppose that there would exist $\epsilon>0$ and a measure $f(\cdot)$ such that for all subsets $D\subseteq\Delta^K_N$,
%$$\inf_{\mu\in\Delta^K_N}\inf_{a\in\calA}\;\Prob\left[\mu_{t}\in D\mid \mu_{t-1}=\mu,a_t=a\right]\geq \epsilon f(D)\enspace.$$
%Here, this condition does not hold since the transitions are deterministic: $\Prob\left[\mu_t=\nu \mid \mu_t=\mu, a_t=a\right] = \mathds{1}_{(\nu = \mu P(a))}$. This constitutes the most degenerate (and difficult) case to handle.
%\end{remark}

%-----
\subsection{Optimality criteria}
We suppose that the controller aims to maximize her average long-term reward, i.e.,
\begin{equation}\label{eq::optimality_criteria}
g^*(\mu_0) = \sup_{\pi\in\Pi} \liminf_{T\to\infty} \frac{1}{T}\sum_{t=1}^{T} r(\pi_t(\mu_t),\mu_t)\enspace,
\end{equation}
where $r(a,\mu) = \int_{\Xi} r(a,\mu,\xi)  d \sigma(\xi)$.
Starting from $\mu_0$, the population distribution will evolve in $\Delta_N^K$ through the dynamics described in~\Cref{eq::linear_transition} according to a policy $\pi\in\Pi$. Nonetheless, with the assumptions we made, we next show that the dynamics effectively evolves on a particular subset of $\Delta_N^K$.

Let $Q^k_L(a_{:L}):=\prod_{l\in[L]} P^k(a_l)$ be the transition matrix over $L$ time steps, and $\calD_L$ be defined as $\calD_L = \bigtimes_{k\in[K]}\calD_L^k$ where
\begin{equation}\label{eq::D_L}
\calD^k_L = \Conv\left(\left\{
\mu^k Q^k_L(a_{:L},\xi_{:L}) \;\left\vert\;
\begin{split}
& a_{:L} \in \calA^L, \\
&\mu^k\in\Delta_N,\\
&\xi_{:L}\in\Xi^L
\end{split}
\right.\right\}\right)\enspace.
\end{equation}
\begin{lemma}\label{lemma::D_L}
Let \ref{hypo::continuity_P}-\ref{hypo::positiveness_P} hold. Then $\calD$ is a compact set included in the relative interior of $\Delta_N^K$. Moreover, for $t\geq 1$, $\mu_t\in \calD$ for any policy $\pi\in\Pi$.
\end{lemma}
\begin{proof}
The set $\{\mu^k Q^k(a_{:L},\xi_{:L}) \mid (a_{:L},\mu^k,\xi_{:L}) \in \calA^L\times\Delta_N\times \Xi^L\}$ is compact, by continuity of $(a,\mu,\xi)\mapsto \mu Q^k(a,\xi)$ and compactness of $\Delta_N$, $\calA$ and $\Xi$. Therefore, $\calD_L$ is compact as it is the convex hull of a compact set in finite dimension. Then, the positiveness of $Q^k$ implies that $\calD_L^k\subset\Relint(\Delta_N)$. Moreover, by property of the semiflow, $ \phi_t(a_{:t},\xi_{:L},\mu_0) = \phi_L\left(a_{t-L+1:t}, \xi_{t-L+1:t},\phi_{t-L}(a_{:t-L},\xi_{:t-L},\mu_0)\right) \in \calD_L$.
\end{proof}

We recall that the relative interior of the simplex, equipped
with Hilbert's projective metric, is a complete metric
space, on which the Hilbert's metric topology is the same
as the Euclidean topology, see~\cite[\S~2.5]{Lemmens_2009}.
Hence, under~\ref{hypo::continuity_P} and~\ref{hypo::primitivity_P}, $(\calD_L,d_H)$ is a complete metric space. We also recall {\em Birkhoff theorem},
which shows that every matrix $Q\gg 0$ is a contraction in Hilbert's projective metric, i.e., for all $\xi,a\in\Xi\times\calA$ and $\mu,\nu \in \calD$,
\begin{equation}\label{eq::contraction}d_H(\mu P(a,\xi),\nu P(a,\xi)) \leq \kappa(P(a,\xi)) d_H(\mu,\nu)\enspace,
\end{equation}
where
\[ \kappa(Q) := \tanh\left(\Diam_H(Q)\,/\,4\right) < 1
\enspace,
\]
see~\cite[Appendix~A]{Lemmens_2009}.
This property applies to the transition matrix $P^k(a)$ under \ref{hypo::positiveness_P}, or to $Q^k_L$ under \ref{hypo::primitivity_P}.

% -----
\subsection{Ergodic eigenproblem}
For any real-valued function $v : \Delta_N^K\to \bbR$ and discount factor $\alpha\in]0,1]$, we define the Bellman operator $\BelOp_\alpha$ as
\begin{equation}
\BelOp_\alpha v(\mu) = \max_{a\in\calA} \int_{\Xi} \left[r(a,\mu,\xi) + \alpha v(\mu P(a,\xi))\right]d\sigma(\xi)\enspace.
\end{equation}
For $\alpha=1$, we simply write $\BelOp = \BelOp_1$. For $\alpha<1$, we denote by $v_\alpha$ the solution of $\BelOp_\alpha v = v$, which can be obtained as the limit of the sequence $(v^j_\alpha)_{j\in\bbN}$ where $v^{j+1}_\alpha = \BelOp_\alpha v^j_\alpha$ and $v^j_\alpha \equiv 0$. This result follows from the fact that the Bellman operator $\BelOp_\alpha$ is a sup-norm contraction, for $\alpha <1$.
A first observation is that $\mu \mapsto (\BelOp v) (\mu)$ is convex for any real-valued convex function $v$. Indeed, the transition dynamics~\eqref{eq::linear_transition} is linear in $\mu$, as well as the reward~\eqref{eq::shape_r}; therefore, for any $a \in \calA$, the expression under the maximum is convex in $\mu$, and since the maximization preserves the convexity, the observation is established. 
For a feedback policy $\pi$, we also define $\BelOp^\pi$ the Kolmogorov operator such that $\BelOp^\pi v\,(\mu) = \int_\Xi \left[r(\pi(\mu),\mu,\xi) + v(\mu P(\pi(\mu),\xi))\right]d\sigma(\xi)$.

\subsubsection{Existence of a solution} As mentioned previously, neither the minorization condition~\cite{Kurano_1989,Hernandez_Lerma_1996,Wiecek_2019} nor the controllability condition~\cite{KM} apply in our situation. Instead, we exploit here the contraction
properties of the dynamics, with respect to Hilbert's
projective metric, together with the vanishing
discount approach, to show the existence. First let us show a preliminary result on metrics comparison:%, to show the following result.

\begin{lemma}\label{lemma::akian}
Let $\calD\subset \Relint(\Delta_n)$, $n\in\bbN$ and $x,y \in\calD$. Then, 
\begin{equation}
n\left\| x - y \right\|_\infty  \leq d_H(x,y) \Upsilon(\Diam_H(\calD))
\end{equation}
where $\Upsilon(d) = \frac{1}{d}e^{d}(e^{d} - 1)$.
\end{lemma}
\begin{proof}
We use the results in~\cite{Akian_2016}: Lemma 2.3 shows that for any vectors $u,x,y\in\calD$ such that there exist $a,b>0$ satisfying $ax\leq u\leq bx$ and $ay\leq u\leq by$, we have the following inequality:
$$\left\|x-y\right\|_u \leq \left(e^{d_T(x,y)} -1\right) e^{\max(d_T(x,u),d_T(y,u))} \enspace,$$
where $d_T$ denotes the Thompson distance, and $\left\|z\right\|_u = \inf\{a>0\mid -au\leq z\leq au\}$. In particular, by choosing $u = (1/n,\hdots 1/n)$ as the center of the simplex, $\left\|\cdot\right\|_u = n\left\|\cdot\right\|_\infty$.
Moreover, $d_T(\cdot,\cdot)\leq d_H(\cdot,\cdot)$ on $\Relint(\Delta_N^K)$, see~\cite[Eq. 2.4]{Akian_2016}.
Therefore,
$$
\begin{aligned}
n\left\|x-y\right\|_\infty &\leq \left(e^{d_H(x,y)} -1\right) e^{\max(d_H(x,u),d_H(y,u))}\\
&\leq \left(e^{d_H(x,y)} -1\right) e^{\Diam_H(\calD)}\enspace.
\end{aligned}
$$
We easily conclude using the fact that $f: x\mapsto e^{x} -1$ is a convex function, and so for all $ 0\leq x \leq \bar{x},\, f(x) \leq x \frac{e^{\bar{x}}-1}{\bar{x}}$.
\end{proof}

Applying~\Cref{lemma::akian}, we obtain that $\mu \mapsto r(a,\mu,\xi)$ is Lipschitz of constant $M^{\calD}_r := \frac{1}{K}M_r \Upsilon(\Diam_H(\calD_L))$ for the Hilbert metric.

Let us define the optimal infinite horizon discounted objective $v_\alpha$, defined as
\begin{equation}\label{eq::Valpha}
v_\alpha(\mu_0) = \sup_{\pi\in\Pi} \;\sum_{t\geq 1} \alpha^{t-1}r(\pi_t(\mu_t),\mu_t)\enspace,
\end{equation}
where $\alpha$ is the discount factor and $\mu_0$ is the initial distribution.
As a consequence of~\Cref{lemma::akian}, we obtain that the value functions
of the discounted problems constitute an equi-Lipschitz family:
\begin{lemma}[Equi-Lipschitz property]\label{lemma::equiLip}
Assume that~\ref{hypo::continuity_P}-\ref{hypo::bounded_r} hold.
Then, $(v_\alpha)_{\alpha\in(0,1)}$ is $\left(\frac{M^{\calD}_r}{1 - \kappa}\right)$-equi-Lipschitz on $\calD_L$ for the Hilbert metric, i.e., for all $\mu_0,\nu_0\in\calD_L$, 
$$
\left\vert v_\alpha(\mu_0) - v_\alpha(\nu_0)\right\vert 
\leq \frac{M^{\calD}_r}{1- \kappa} d_H(\mu^0,\nu^0)\enspace .
$$
\end{lemma}

\begin{proof}
We first make the proof under the stronger assumption~\ref{hypo::positiveness_P}, and then deduce the general result.

We denote by $(v^j_\alpha)_{j\in\bbN}$ the sequence defined as $v^{j+1}_\alpha = \BelOp_\alpha v^j_\alpha$ and $v^0_\alpha \equiv 0$.
Let us assume that for a given $j\in\bbN$, $v^j_\alpha$ is $M^j_\alpha$-Lipschitz w.r.t the Hilbert metric , i.e., 
$$
\left\vert v^j_\alpha(\mu) - v^j_\alpha(\nu) \right\vert \leq M^j_\alpha d_H(\mu,\nu)\enspace.
$$
Then, for $\mu,\nu\in\calD_1\subset\Delta_K^N$, we have:
$$
\begin{aligned}
&\left\vert\BelOp_\alpha v^j_\alpha(\mu) - \BelOp_\alpha v^j_\alpha(\nu)\right\vert \\
&\leq \int_{\Xi} \left\vert
r(a,\mu,\xi) - r(a,\nu,\xi)
\right\vert
d\sigma(\xi) \\
&\quad + \alpha\int_{\Xi} \left\vert
v^j_\alpha(\mu P(a,\xi)) - v^j_\alpha(\nu P(a,\xi))
\right\vert
d\sigma(\xi)\\
& \leq M^{j+1}_\alpha d_H(\mu,\nu),
\end{aligned}
$$
with $
M^{j+1}_\alpha = M^{\calD}_r + \alpha\kappa M^j_\alpha
$. 
Therefore, for all $j\in\bbN$, $M^j_\alpha \leq M: = \frac{M^{\calD}_r}{1 - \kappa}$, which is independent of $j$ and $\alpha$. So, at the limit, $v_\alpha$ is $M$-equi-Lipschitz w.r.t. the Hilbert pseudo-metric.

To deduce the general result with~\ref{hypo::primitivity_P}, we define
\begin{itemize}
  \item $\widetilde{\calA}:= \calA^L$, $\widetilde{\Xi}:= \Xi^L$, $\tilde{\alpha}:=\alpha^L$,
  \item $\tilde{\phi}_\tau(\tilde{a}_{:\tau},\tilde{\xi}_{:\tau},\mu_0):=\mu_0 \prod_{1\leq t\leq \tau}Q(\tilde{a}_t,\tilde{\xi}_t)$,
  \item $\tilde{r}(a_{:L},\mu,\xi_{:L}):= \sum_{l\in[L]}\alpha^{l-1}r(a_l,\phi_l(a_{:l},\mu,\xi_{:l}),\xi_l)$,
  \item and $$\tilde{\BelOp}_\alpha v(\mu) = \max_{\tilde{a}\in\widetilde{\calA}} 
\left\{\begin{split}
&\int_{\tilde{\Xi}}\left[\tilde{r}(\tilde{a},\mu,\tilde{\xi})+v(\nu)\right]d\tilde{\sigma}(\tilde{\xi})\\
&\text{s.t} \;\nu = \mu Q_L(\tilde{a},\tilde{\xi})
\end{split}\right\}$$.
\end{itemize}
and observe that 
$$v_\alpha(\mu_0) =  \sum_{\tau\geq 1}\tilde{\alpha}^{\tau-1} \tilde{r}(\tilde{a}_\tau,\tilde{\phi}_\tau(\tilde{a}_{:\tau},\tilde{\xi}_{:\tau},\mu_0))\enspace.$$
We have rescaled the time ($\tau$ instead of $t$) so that the transition matrix between time $\tau$ and time $\tau+1$ is $Q_L(\tilde{a}_\tau,\tilde{\xi}_\tau)$. One $\tau$-time step corresponds to $L$ $t$-time steps.
As the transition $Q_L(\tilde{a},\tilde{\xi})$ is now positive, the proof is exactly the same as before, in the $\tau$-time space. 
\end{proof}

\begin{remark}
The result remains if the common noise $\xi_t$ is controlled, i.e., depends on the action $a_t$. However, if $\xi_t$ depends on the state, then there is no guarantee that $M^{j+1}_\alpha$ is bounded. In the latter case, it would require $\left\|v^j_\alpha\right\|_\infty$ to be uniformly bounded, which is not guaranteed (when $\alpha\to 1$, it goes to infinity in general).
\end{remark}

We are now able to prove the main result:
\begin{theorem}[Existence of a solution]\label{prop::ergodic_fixed_point}
  Assume that~\ref{hypo::continuity_P}-\ref{hypo::bounded_r} hold.
  Then, the ergodic eigenproblem
\begin{equation}\label{eq::ergodic_fixed_point}
g \OneF_{\calD_L} + h = \BelOp h
\end{equation}
admits a solution $h^*\in \Lip(\calD_L)\cap \vex(\calD_L)$ and $g^*\in\bbR$.
\end{theorem}
\begin{proof}
Let us define a reference distribution $\overline{\mu} \in \Delta_N^K$, $g^*_\alpha=(1-\alpha)v_\alpha(\overline{\mu})$ and $h^*_\alpha = v_\alpha - v_\alpha(\overline{\mu}) \OneF_{\calD_L}$. Then, as $v_\alpha$ is equi-Lipschitz on $\calD_L$ (\Cref{lemma::equiLip}), $h^*_\alpha$ is equi-bounded and equi-Lipschitz on $\calD_L$ (in particular equi-continuous).
 By the Arzelà-Ascoli theorem, $h^*_\alpha \to h^* \in \Conti(\calD_L)$.

Finally, from the discounted reward approach, we get $\BelOp(\alpha v_\alpha) = v_\alpha$, 
therefore 
$$\frac{g^*_\alpha}{1-\alpha}\OneF_{\calD_L} + h^*_\alpha = \BelOp\left(\frac{\alpha g^*_\alpha}{1-\alpha}\OneF_{\calD_L} + \alpha h^*_\alpha\right)\enspace.$$ By the additive homogeneity property of the Bellman function, $g^*_\alpha \OneF_{\calD_L} + h^*_\alpha = \BelOp(\alpha h^*_\alpha)\enspace.$ The fixed-point equation~\eqref{eq::ergodic_fixed_point} is then obtained by continuity of the Bellman operator $\BelOp$.
  
To conclude, $h^*$ is convex since  $v_\alpha$ is convex and the pointwise convergence preserves the convexity.
\end{proof}

\begin{prop}\label{prop::verification_theorem}For any solution $(g^*,h^*)$ of~\eqref{eq::ergodic_fixed_point}, $g^*$ satisfies~\eqref{eq::optimality_criteria}, and
a maximizer $a^*(\cdot) \in \argmax \BelOp h^*$ defines an optimal \emph{stationary} policy for the average gain problem.
\end{prop}
\begin{proof}
Let $\pi \in \Pi$ be a policy.
By definition, for every $t$, $\BelOp^{\pi_t} h^* \leq \BelOp h^* = h^*+g^*\OneF_{\calD_L}$. Therefore, iterating the Kolmogorov operator, we obtain
$$(\BelOp^{\pi_1} \circ \hdots \circ \BelOp^{\pi_t}) h^* \leq h^* + t g^*\OneF_{\calD_L} \enspace.$$ 
Let $\underline{h}^* := \min_{\mu\in\calD_L} h^*(\mu)$ be the minimum of $h^*$. Then, $0_{\calD_L}\leq h-\underline{h}^*\OneF_{\calD_L}$, and so
$(\BelOp^{\pi_1} \circ \hdots \circ \BelOp^{\pi_t}) (0_{\calD_L}) \leq h^* + (t g^* -\underline{h}^*) \OneF_{\calD_L} \enspace.$
Finally, $$\liminf_{t\to\infty} \frac{1}{t}(\BelOp^{\pi_1} \circ \hdots \circ \BelOp^{\pi_t}) (0_{\calD_L}) (\mu_0) \leq g^*\enspace.$$
Any strategy has an average reward lower than $g^*$. As we have proved that the bias function $h^*$ is continuous on $\calD_L$, a maximizer $a^*(\mu)$ can be found for any state $\mu$, and so playing the strategy $a^*(\mu)$ achieves the best possible average gain $g^*$.
\end{proof}
In particular, the constant $g^*$ in~\eqref{eq::ergodic_fixed_point}
is unique, and it coincides with the optimal average long-term reward,
for all choices of the initial state $\mu_0$. However, even if the payoff $g^*$ is unique, the bias function $h^*$ is not (and so neither is the optimal policy).

\begin{remark}
In~\cite{Bauerle_2023}, B\"auerle also used the vanishing discount approach, but here we do not assume a priori the equi-boundedness of the optimal discounted objective functions $v_\alpha$. Instead, using a contraction argument on the dynamics, we obtained that $(v_\alpha)_{\alpha\in(0,1)}$ is equi-Lipschitz (see~\Cref{lemma::equiLip}). In particular, it entails that any optimal eigenvector $h^*$ is Lipschitz (and not only upper semi-continuous).
\end{remark}

\emph{In the sequel, we restrict the study to deterministic problem (absence of common noise, i.e., $\Xi$ is reduced to a singleton).}

\subsubsection{Non-uniqueness of the solution}
As discussed in the introduction, classical approaches to the infinite dimensional ergodic problem rely on a geometric ergodicity/Doeblin type condition. This condition entails that the bias function is unique up to an additive constant. We next show that under~\ref{hypo::continuity_P} to~\ref{hypo::bounded_r}, the bias function may not be unique, implying that the present results cannot be derived from such approaches.

%% In contrast with results where geometric ergodicity is assumed to guarantee the existence (and the uniqueness) of the eigenvector, see e.g.~\cite{Biswas_2015,Hernandez_Lerma_1996}, the uniqueness of the latter is generally not true.
To get a non-unique bias,  we will construct instances where there exist several ``attractor'' states, and where a family of strategies can be found so that each of them secures the optimal mean payoff. Then, different attractors lead to different bias functions.
%% In these instances, the eigenvector does exist but is not unique.
To illustrate this fact, we introduce in~\Cref{ex::non_uniqueness} a deterministic model satisfying~\ref{hypo::continuity_P} to~\ref{hypo::bounded_r}. Note that taking the same dynamics as in the example but without node 2 can also lead to a non-unique solution of the eigenproblem as long as we allow for a more general form of reward $r(a,\mu)$. Here, we aim at fitting exactly with our application case by considering that the reward function satisfies~\cref{eq::shape_r}.
\begin{example}[Non-uniqueness of the eigenvector]\label{ex::non_uniqueness}
\begin{center}
\begin{tikzpicture} [shorten >=1pt, auto,
    node distance=3cm, scale=0.5, 
    transform shape, align=center, 
    state/.style={circle, draw, minimum size=1cm}]
\node (q1) [state] {\Huge $1$};
\node (q2) [state, right = of q1] {\Huge $2$};
\node (q3) [state, right = of q2] {\Huge $3$};
\path [-stealth, thick]
    (q1) edge [bend left=15] node [above=0.1cm] {\huge $a$} (q2)
    (q2) edge [bend left=15] node [below=0.1cm] {\huge $1-a$} (q1)
    (q2) edge [bend left=15] node [above=0.1cm] {\huge $a$} (q3)
    (q3) edge [bend left=15] node [below=0.1cm] {\huge $1-a$} (q2)
    (q1) edge [loop left] node[left=0.2cm] {\huge $1 - a$}    (   )
    (q3) edge [loop right] node[right=0.2cm] {\huge $a$}   (   );    
\end{tikzpicture}
\end{center}
Let us consider the dynamical system described by the following transition matrix:
$$P(a)=\begin{bmatrix}
1-a & a & 0\\
1-a & 0 & a\\
0 & 1-a & a
\end{bmatrix},$$
where the $a$ is supposed to belong to
the action space $\calA$, which is of the form
$$\calA = [a_0,a_1]\;,\quad 0 < a_0 < 1/2 \;\text{ and }\; a_1 = 1 - a_0\enspace .$$
We consider the following unitary reward $\theta(\cdot)$: 
$$\theta(a)_n = \begin{cases}
    1-a \text{ if } n = 1,\\
    0\;\; \text{ if } n = 2,\\
    a\;\; \text{ if } n = 3,
\end{cases}$$
The reward function is then $r(a,\mu) = (1-a)\mu_1 + a \mu_3$ for any $\mu\in\Delta_3$, and 
$$r(a,\mu P(a)) = (1-a)^2 (1 - \mu_3) + a^2(1-\mu_1)\enspace.$$
\end{example}

In the sequel, we work in the sub-simplex $\Delta^{\leq}_3 := \{(x,y)\in\bbR^2_{\geq 0} \mid x+y\leq 1\}$, considering that $\mu_2$ can be reconstructed as $\mu_2 = 1-\mu_1 - \mu_3$.

% ---
\paragraph{The associated ergodic eigen problem.} For any real-valued function $v: \Delta^{\leq}_3 \to \bbR$, let us define the Bellman operator $\BelOp$ as 
$$\BelOp v(\mu_1,\mu_3) = \max_{a\in\calA} \left\{\begin{split} 
&(1-a)^2 (1 - \mu_3) + a^2(1-\mu_1) \\
&+ v\left((1-a)(1-\mu_3), a(1-\mu_1)\right)
\end{split}
\right\}\enspace.$$
In~\Cref{ex::non_uniqueness}, the transition $a\in\calA \mapsto P(a)$ is linear. 
Moreover, the transition matrix over two time steps is then
$$(P(a))^2=\begin{bmatrix}
(1-a)^2+a(1-a) & a(1-a) & a^2\\
(1-a)^2 & 2a(1-a) & a^2\\
(1-a)^2 & a(1-a) & a^2 + a(1-a)
\end{bmatrix}
$$
and has positive coefficients. Therefore, the transition matrix $P(a)$ satisfies the primitivity assumption~\ref{hypo::primitivity_P} for all $a\in\calA$.
Using~\Cref{prop::ergodic_fixed_point}, the ergodic eigenproblem
\begin{equation}\label{eq::BelOp_example}
g 1_{\calD_1} + h = \BelOp h
\end{equation}
admits a solution $h^*\in \Lip(\calD_1)\cap \vex(\calD_1)$ and $g^*\in\bbR$, where $\calD_1$ is defined one can construct the effective domain $\calD_1$ as in~\eqref{eq::D_L}.
As the quantity in the maximum is convex, for any convex function $v :\calD_1 \to \bbR$, the maximum value in $\BelOp v$ is obtained for $a=a_0$ or $a=a_1$.
Therefore, in the sequel, we restrict wlog the state space to be $\calA = \{a_0,a_1\}$.

\paragraph{Steady states.} 
Let $k\in\{0,1\}$. The equilibrium distribution $\hat{\mu}$ achieved by a constant decision $a_k$ is given by the equation $\hat{\mu} P(a_k) = \hat{\mu}$, which has a unique solution:
\begin{equation}\label{eq::steady_states}
    \hat{\mu}^k_1 = \frac{(1-a_k)^2}{1-a_k(1-a_k)}, \; \hat{\mu}^k_3 = \frac{a_k^2}{1-a_k(1-a_k)} \enspace.
\end{equation}

\paragraph{Bias function for the Kolmogorov operator.} Let us define $\BelOp^k$ the Kolmogorov operator associated to the constant strategy $\pi:\mu\mapsto a_k$, i.e.,
$$
\begin{aligned}
\BelOp^k v(\mu_1,\mu_3) =& (1-a_k)^2 (1 - \mu_3) + a_k^2(1-\mu_1) \\
&+ v\left((1-a_k)(1-\mu_3), a_k(1-\mu_1)\right)\enspace.
\end{aligned}$$
Then, the linear function $h^k(\mu_1,\mu_3) = \alpha^k \mu_1 + \beta^k \mu_3$ and the gain $g^k$ are solutions of 
\begin{equation}\label{eq::Kolmo}
h^k(\mu_1,\mu_3) + g^k = \BelOp^k h^k(\mu_1,\mu_3), \; (\mu_1,\mu_3)\in\calD_1
\end{equation}
if and only $g^k$, $\alpha^k$ and $\beta^k$ satisfy the following system
$$\left\{\begin{aligned}
g^k &= (1-a_k)^2 + a_k^2 + (1-a_k)\alpha^k + a_k\beta^k \\
\alpha^k &= - a_k^2 - a_k\beta^k\\
\beta^k &= - (1-a_k)^2 - (1-a_k)\alpha^k\\
\end{aligned}
\right.\enspace,$$
where the unique solution of the latter system is given by 
\begin{equation}\label{eq::sol_kolmo}
\begin{aligned}
g^k &= \frac{a_k^3 + (1-a_k)^3}{1 - a_k(1-a_k)},\\
\alpha^k &= \frac{a_k(1-a_k)^2 - a_k^2}{1 - a_k(1-a_k)},\\
\beta^k &= \frac{(1-a_k)a_k^2 - (1-a_k)^2}{1 - a_k(1-a_k)}\enspace.
\end{aligned}
\end{equation}
Note that $g^0 = g^1$ since $a_0 +a_1 = 1$, and we simply denoted it by $g^*$.

\paragraph{Solution for the ergodic eigen problem.} We now exhibit a family of eigenvectors where each of them constitutes a solution to the ergodic eigenproblem associated with~\Cref{ex::non_uniqueness}:
\begin{theorem}[Non-uniqueness of the eigenvector]
Let $v^k:\calD\to\bbR$, $k\in\{0,1\}$, be defined as
\begin{equation}
v^k(\mu_1,\mu_3) = \hat{h}^{k0}(\mu_1,\mu_3)\vee \hat{h}^{k1}(\mu_1,\mu_3) \vee \hat{h}^{k2}(\mu_1,\mu_3)\enspace,
\label{eq::bias_vk}
\end{equation}
with
\begin{itemize}
\item[$\diamond$] $\hat{h}^{ij}(\cdot,\cdot) := h^{j}(\cdot,\cdot) - h^{j}(\hat{\mu}^i_1, \hat{\mu}^i_3), \;i,j \in\{0,1\}$,
\item[$\diamond$] $\hat{h}^{k2}(\cdot, \cdot) := \BelOp^{1-k} \hat{h}^{kk} (\cdot,\cdot) - g^*$.
\end{itemize}
Then, for any $\lambda\in[0,1]$, the couple $(v^{\lambda},g^*)$ is solution the ergodic eigenproblem~\eqref{eq::BelOp_example} -- corresponding to~\Cref{ex::non_uniqueness} -- with $g^*$ defined in~\eqref{eq::sol_kolmo} and 
$$v^{\lambda}(\mu_1,\mu_3) := \left(v^0(\mu_1,\mu_3) - \tfrac{\lambda}{1-\lambda}\right) \vee \left(v^1(\mu_1,\mu_3) - \tfrac{1-\lambda}{\lambda}\right)\enspace.$$
\end{theorem}
\begin{proof}
As first observation, the couple $(h^k,g^*)$, solution of~\eqref{eq::Kolmo}, is not solution of the ergodic eigenproblem~\eqref{eq::BelOp_example}. 
Therefore, let us try to construct a solution as a mixture  of $h^0$ and $h^1$. To this purpose, let us define the function the function $u^0: \calD_1 \to \bbR$ as 
$$
u^0(\mu_1,\mu_3) = \hat{h}^{00}(\mu_1,\mu_3)\vee \hat{h}^{01}(\mu_1,\mu_3) \enspace.
$$
For $(\mu_1,\mu_3)\in\calD_1$, the value of $\BelOp u^0(\mu_1,\mu_3)$ is given by the maximum of 4 quantities:
\begin{enumerate}[label=(\roman*)]
    \item $\BelOp^0 \hat{h}^{00} (\mu_1,\mu_3) = \hat{h}^{00} (\mu_1,\mu_3) + g^*$,
    \item $\BelOp^0 \hat{h}^{01} (\mu_1,\mu_3) $,
    \item $\BelOp^1 \hat{h}^{00} (\mu_1,\mu_3) $,
    \item $\BelOp^1 \hat{h}^{01} (\mu_1,\mu_3) = \hat{h}^{01} (\mu_1,\mu_3) + g^*$,
\end{enumerate}
The equality in (i) and (iv) comes from the fact that $\hat{h}^{00}$ and $\hat{h}^{01}$ are solutions for~\eqref{eq::Kolmo}.
Besides, we can prove using basic algebra that 
$$\BelOp^0 \hat{h}^{01} (\mu_1,\mu_3) - \BelOp^1 \hat{h}^{00} (\mu_1,\mu_3) = g^*(\hat{\mu}^0_3 - \hat{\mu}^0_1)\leq 0\enspace.$$
Therefore, the maximum is obtained either with (i), (iii) or (iv). We consider now the function
\begin{equation}
v^0(\mu_1,\mu_3) = \hat{h}^{00}(\mu_1,\mu_3)\vee \hat{h}^{01}(\mu_1,\mu_3) \vee \hat{h}^{02}(\mu_1,\mu_3)\enspace,
\label{eq::bias_v0}
\end{equation}
with $\hat{h}^{02}(\mu_1,\mu_3) := \BelOp^1 \hat{h}^{00} (\mu_1,\mu_3) - g^*$.
By construction, $v^0 = \BelOp u^0 - g^* $.
Moreover, one can show that $\BelOp \hat{h}^{02}(\mu_1,\mu_3) - g^*\leq v^0 (\mu_1,\mu_3)$. 
Therefore, for all $(\mu_1,\mu_3)\in\calD_1$,
$$\BelOp v^0(\mu_1,\mu_3) = \BelOp u^0(\mu_1,\mu_3) = v^0(\mu_1,\mu_3) - g^*\enspace.$$
As a conclusion, $(v^0,g^*)$ is a solution of~\eqref{eq::BelOp_example}.

By symmetry of the problem, we can construct the function $v^1(\mu_1,\mu_3) = v^0(\mu_3,\mu_1)$, and $(v^1,g^*)$ is a different solution of~\eqref{eq::BelOp_example}. Finally, each max-plus combination of $v^0$ and $v^1$ also constitutes a solution of the ergodic eigenproblem.
\end{proof}

We display in~\Cref{fig::ex_non_uniqueness} the eigenvector $v^0$, $v^{1/2}$ and $v^1$, obtained numerically (using the RVI procedure, see~\Cref{algo::RVI}), as with the eigenvector $v^0$, obtained theoretically (see above).

%\begin{remark}
%We continue the analysis of~\Cref{ex::non_uniqueness} in~\Cref{sec::extensions}, where the convergence of the dynamical systems under optimal policies is studied through the weak KAM angle, showing in particular that the projected Aubry set is reduced to two points, corresponding to the states defined in~\eqref{eq::steady_states}. From this result, we will prove that any eigenvector is a tropical linear combination of the two basis eigenvectors $v^0$ and $v^1$, showing that the family $\{v^{\lambda}\}_{\lambda\in[0,1]}$ describes all the eigenvectors that are solutions of the ergodic eigenproblem.
%\end{remark}

\begin{figure}[!ht]
    \centering
    \includegraphics[width=0.99\linewidth]{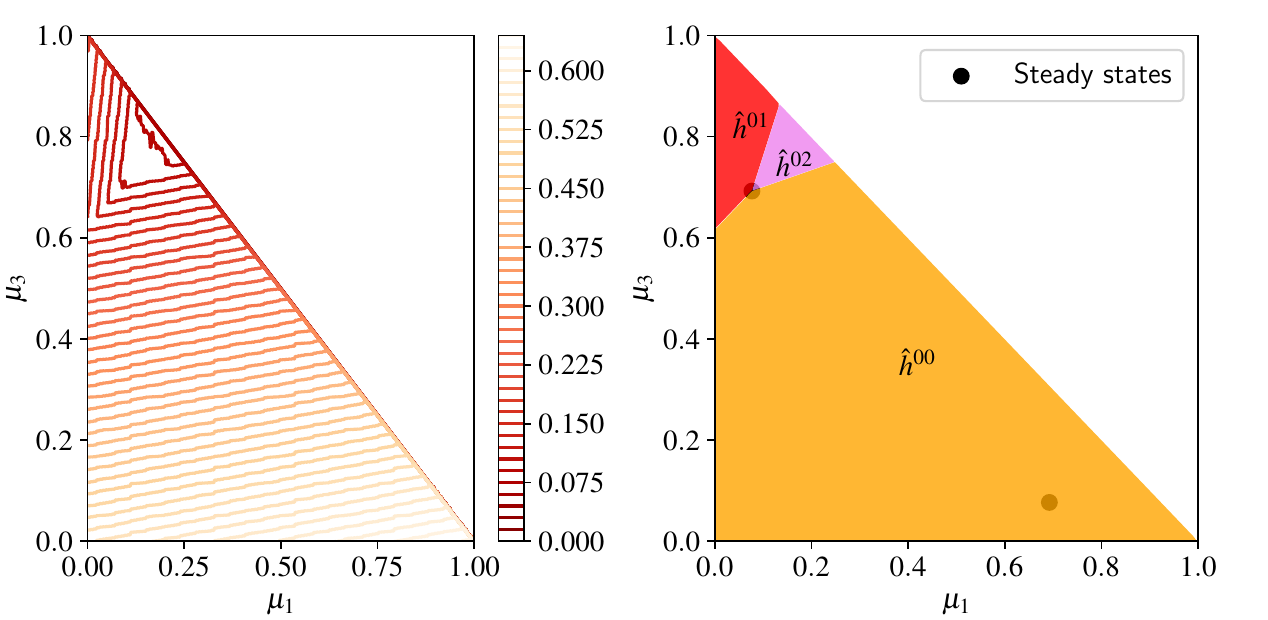}
    \includegraphics[width=0.99\linewidth]{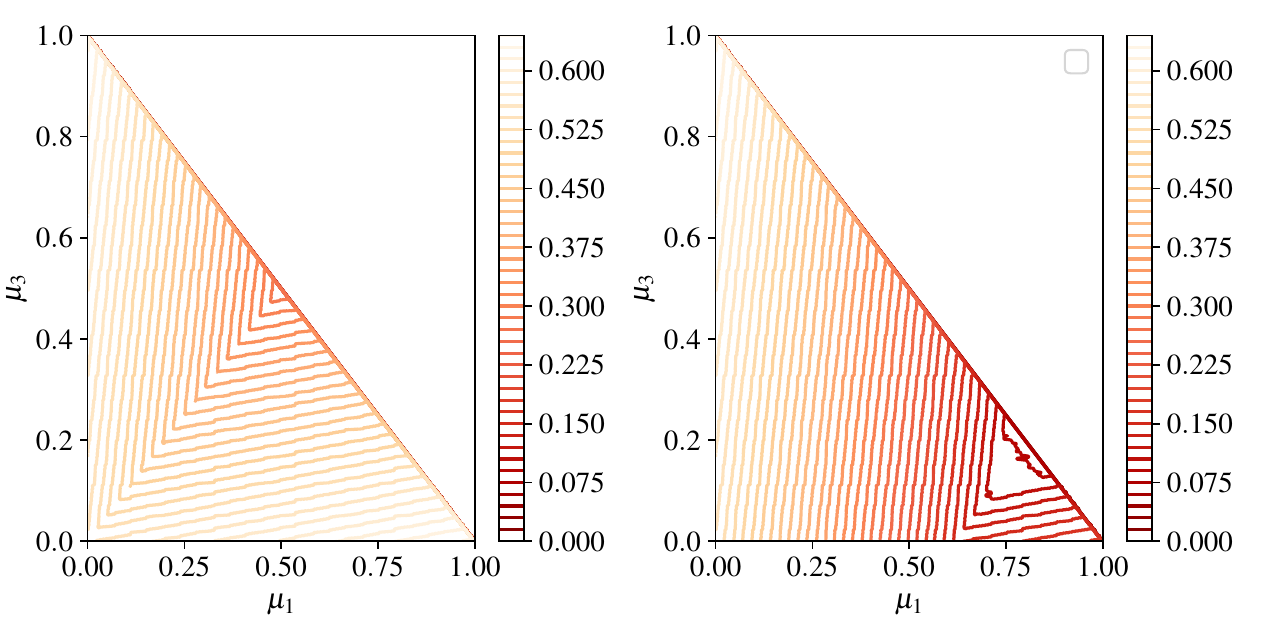}
    \caption{Eigenvectors for~\Cref{ex::non_uniqueness} with $\calA = [0.25,0.75]$.\\
    (Upper left): the eigenvector $v^0$ (obtained by the RVI algorithm, see~\Cref{sec::resolution}). (Upper right): the theoretical $v^0$ found in~\eqref{eq::bias_v0}, showing the two steady states to which each of optimal strategies converges. (Lower left): the eigenvector $v^{1/2}$. (Lower right): the eigenvector $v^1$. }
    \label{fig::ex_non_uniqueness}
\end{figure}

% =============================================
\section{Numerical resolution}\label{sec::resolution}
We present in this section two iterative algorithms in order to numerically solve the ergodic eigenproblem~\eqref{eq::ergodic_fixed_point}.
\subsection{Relative Value Iteration with Krasnoselskii-Mann damping}
Relative Value Iteration (RVI) has been extensively studied to solve unichain finite-state MDP~\cite{Puterman_1994,Bertsekas_1998}.
Simplicial state-spaces appear in particular in the definition of \emph{belief state} for partially observable MDP~\cite{Hauskrecht_2000}. For such continuous state-spaces, a discretization must be done as a prerequisite to RVI algorithm. Here, we define a regular grid $\Sigma$ of the simplex $\Delta_N^K$, and $\BelOp^\Sigma$ the Bellman Operator with a linear point approximation on the grid $\Sigma$, achieved by a Freudenthal triangulation~\cite{Lovejoy_1991}. 
With this framework, we have the following property:
\begin{prop}[\cite{Hauskrecht_2000}, Thm 12] For any $v\in\vex(\Delta_N^K)$, $$\BelOp v \leq \BelOp^\Sigma v\enspace.$$
\end{prop}
As the bias function $\hat{h}$ is convex at each iteration, the solution return by~Algorithm~\ref{algo::RVI} provides a gain which is an upper bound of the optimal gain $g^*$.

%\subsection{Average gain}
\begin{algorithm}[!ht] 
\small
\caption{RVI with Mann-type iterates} \label{algo::RVI}
\begin{algorithmic}[1]
%\SetAlgoLined
\Require Grid $\Sigma$, Bellman operator $\BelOp^\Sigma$, initial function $\hh_0$
%\State $v_{max} \leftarrow -\infty$
\State Initialize $\hh = \hh_0$, $\hh'(\mu) = \BelOp^\Sigma \hh$ 
\While{$\Span(\hh'-\hh)>\epsilon$}
\State $\hh \gets (\,\hh'- \max\{\hh'\}e+ \hh\,)/2$
\State $\hh'(\hat{\mu}) \gets (\BelOp^\Sigma \hh)(\hat{\mu})$ for all $\hat{\mu}\in \Sigma$
\EndWhile
\State $\hg \gets (\,\max(\hh'-\hh)+\min(\hh'-\hh)\,)/2$
\State\Return $\hg,\hh$
\end{algorithmic}
\end{algorithm}
In~Algorithm~\ref{algo::RVI}, we use, following~\cite{Gaubert_2020},
a mixture of the classical relative value iteration
algorithm~\cite{Puterman_1994} with a \emph{Krasnoselskii-Mann} damping.
As detailed in~\cite{Gaubert_2020} (Th.~9 and Coro~13), it follows
from a theorem of Ishikawa that the sequence of bias function $\hat{h}$ does converge, and it follows from a theorem of Baillon and Bruck that $\hat{g}$ provides an $\epsilon$ approximation of the optimal average cost $g^*$
after $O(1/\epsilon^2)$ iterations.
%% Moreover, a theorem
%% of Ishikawa~\cite{ishikawa} entails the sequence of bias function $h$ does converge.
%% aloinin which $\hh\leftarrow\frac{1}{2}(\hh+\BelOp^\Sigma \hh)$ instead of the usual so-called \emph{Picard iterates} $\hh\leftarrow \BelOp^\Sigma \hh$, see for instance~\cite[Chapter 6]{Chidume_2009}. This idea is adapted to relative value iteration by working
%% in the Banach space of continuous functions, modulo constants.
%% A convergence analysis, which also applies to the infinite dimensional setting,
%% can be found~\cite{Gaubert_2020}. %\todo{SG: explicit error estimate to beadded more appealing}

\subsection{Howard algorithm with on-the-fly transition generation}\label{sec::Howard}
We focus here on an other class of iterative methods to solve MDPs, namely \emph{policy iteration} (PI) algorithms, initiated by Howard (see e.g~\cite{Denardo_1968} of \cite{Puterman_1994}).
For deterministic Markov decision processes, a combinatorial implementation of Howard algorithm, with a linear-time per policy,
was given in~\cite{Cochet_1998}.
We refine the latter algorithm, with a method adapted to ``decomposable" state spaces.

Let $\Lambda = (\hat{\mu}_i)_{i\in[M]}$ be a \emph{local} semi-Lagrangian discretization of the simplex $\Delta_N$ of size $M:=|\Lambda|$. We refer the grid to be \emph{local}, since the discretization is done for the probability space of one sub-population and not on the \emph{global} probability space $\Delta^K_N$. The \emph{global} discretization is then 
$$\Sigma = (\hat{\mu}_{\vec{i}_1},\hdots,\hat{\mu}_{\vec{i}_K})_{\vec{i}\in[M]^K}\enspace.$$ 
We define the local transition operator $T^{\Lambda,k}: (i,a)\in [M] \times \calA \mapsto \argmin_{j\in[M]} \|\hat{\mu}_i P^k(a) - \hat{\mu}_j\|_{\infty}$. For each $k\in[K]$, this operator can be computed in a preprocessing step, and stored in $O(M\times |\calA|)$. Note that contrary to the RVI algorithm -- where a Freudenthal triangulation is performed during the computation of $\BelOp^\Sigma$ -- the transition operator is here approximated by finding the closest discretization point (in the $L_\infty$-norm) to the real next state.

The \emph{global} transition can then be obtained \emph{on-the-fly}, i.e., for any action $a\in\calA$ and global index $\vec{i}\in[M]^K$, $T^\Sigma(\vec{i},a)$ can be recomputed whenever it is required in the algorithm knowing the sub-transition $T^{\Lambda,k}(\vec{i}_k,a)$ for all $k\in\bbN$ :
\begin{equation}
    T^{\Sigma}: (\vec{i},a)\in[M]^K\times \calA \mapsto (T^{\Lambda,k}(\vec{i}_k,a))_{k\in[K]}\enspace.
\end{equation}
\begin{remark}
A complete storage of $T^\Sigma$ would lead to a memory occupation in $O(M^K \times|\calA|)$, whereas the storage of all $T^{\Lambda,k}$, $k\in[K]$, is in $O(K\times M\times |\calA|)$.
\label{remark::memory}
\end{remark}

\begin{algorithm}[!ht] 
\small
\caption{Howard Algorithm with on-the-fly transition generation} \label{algo::Howard}
\begin{algorithmic}[1]
%\SetAlgoLined
\Require Local grid $\Lambda$, family of local transitions $(T^{\Lambda,k})_{k\in[K]}$, initial decision vector $\hat{d}'$
\Do
\State $\hat{d} \gets \hat{d}'$
\State $\hg,\hh$ solution of 
\Comment{Policy Evaluation}
$$
\left\{
\begin{aligned}
& \hg + \hh_{\vec{i}} = r(\hat{d}_{\vec{i}},\hat{\mu}_{\vec{i}}) + \hh_{\vec{j}},\; \vec{i}\in\Sigma\\
&\vec{j} = T^{\Sigma}(\vec{i},\hat{d}_{\vec{i}})
\end{aligned}
\right.
$$
\For{$\vec{i}\in\Sigma$}\Comment{Policy Improvement}
\State $\hat{d}'_{\vec{i}} \gets \argmin_{a\in\calA} \left\{
\begin{aligned}
&r(a,\hat{\mu}_{\vec{i}}) + \hh_{\vec{j}} \\
&\mathrm{s.t. } \vec{j} = T^{\Sigma}(\vec{i},a)
\end{aligned} \right\}$
\EndFor
\doWhile{$\hat{d}'\neq \hat{d}$}
\State\Return $\hg,\hat{d}$
\end{algorithmic}
\end{algorithm}

\Cref{algo::Howard} shows the Howard algorithm with on-the-fly transition generation. It consists in alternating a policy evaluation step with a policy improvement step. We implemented a parallelized version of this algorithm\footnote{Available at \url{https://gitlab.com/these_tarif/ergodic_inertia}} by adapting the code of~\cite{Cochet_1998}, 
%\footnote{Available at \url{http://www.cmap.polytechnique.fr/~gaubert/HOWARD2.html}}
initially intended for computing spectral elements in max-plus algebra. The algorithm is known to have experimentally a superlinear convergence which, in finite action-space setting, is reached in finitely many steps, see e.g.\cite{Puterman_1994}. Despite the decomposable transition, all the subpopulations $k\in[K]$ are linked together through a common policy. In the implementation, both the policy $\hat{d}$ and the bias function $\hh$ depend on the global state associated to index $\vec{i}\in [M]^K$. Therefore, the memory needed to run the algorithm is still exponential in the number of segments -- in $O(M^K)$ -- but would have been worst with stored global transition $T^\Sigma$ -- in $O(M^K \times |\calA|)$ -- the action space being very large in general, see~\Cref{remark::memory}. We provide in~\Cref{table::Howard} below benchmarks showing the gain (speedup and memory usage) brought by this approach. 

%%%%%%%%%%%%%%%%%%%%%%%%%%%%%%%%%%%%%%%%%%%%%%%%%%%%%%%%%%%%%%%%%%%%%%%%%%%%%%%%
\section{Steady-state optimality}\label{sec::lagrangian}
\subsection{Definition}
It is of interest to investigate cases in which the dynamic problem reduces to a static one. In fact, in some cases the optimal stationary policy may be a simple policy that attracts the system to a steady-state (``get there, stay there" -- \cite{Flynn_1979}). For instance, Bauerle~\cite{Bauerle_2023} derives a class of mean-field MDPs solvable by a static program. 
%% This idea is even more appealing in the case of continuous state space, where the backward resolution is even more intractable.  

\begin{defin}
Let $\calS=\{(a,\mu)\in \calA\times\Delta_N^K \,\vert\,\mu = \mu P(a)\}$ be the action-space domain of stationary probabilities.
Then, $\mu \in\Delta_N^K$ is a \emph{steady-state} if there exists $a\in \calA$ such that $(a,\mu)\in \calS$.
\end{defin}

%In this section, we look at a specific policy consisting in proposing prices that are constant over time i.e., for any segment $s$, $P_s(x_t) = P_s$ for any $t>1$.
%Since each segment $s$ \emph{independently} evolves according to its own transition matrix $P_s$, we forget in this section the index $s$ for clarity.

%[Existence and uniqueness of the stationary probability] 
If~\ref{hypo::primitivity_P} holds, then for any cluster $k$ and any price $a\in \calA$,
the Markov chain induced by the transition matrix $P^k(a)$ has a unique stationary distribution. We denote by $\overline{\mu}(\cdot): \calA \mapsto \Delta_N^K$ the mapping sending an action to the stationary distribution it induces.

\begin{defin}
The \emph{optimal steady-state gain} $\overline{g}$ is defined as 
\begin{equation}\label{eq::g_infty}
\overline{g} := \max_{(a,\mu)\in  \calS} r(a,\mu)\enspace.
\end{equation}
\end{defin}
If~\ref{hypo::primitivity_P} holds, \eqref{eq::g_infty} is in general a static nonconvex maximization problem over the actions. Nonetheless, we can expect to solve it efficiently in the case where $\overline{\mu}(\cdot)$ is analytically known, see e.g.~Section~\ref{sec::case_study}.
Maximizers $\overline{a}$ are called \emph{optimal steady-state price}, they correspond to a steady-state distribution $\overline{\mu}(\overline{a})$. 

%A first property is that a constant action $\overline{a}$ over time constitutes a valid policy and produces a long-term reward $\overline{g}$:
%\begin{prop}\label{prop::lower_bound}With $(g^*,h^*)$ solution of~\eqref{eq::ergodic_fixed_point} and $\overline{g}$ defined in~\eqref{eq::g_infty}, we have $$\overline{g} \leq g^*\enspace.$$
%\end{prop}
%The proof is given in Appendix~\ref{app::proof}.

% -----
\subsection{Optimality gap}
In this section we introduce a class of Lagrangian functions designed so that each dual problem turns out to be an upper bound of $g^*$. This extends the result of~\cite{Flynn_1979} involving usual Lagrangian functions. We use here a more general Lagrangian, depending on the choice of a non-linear function $\varphi$.
This leads to much tighter bounds, and allows us to prove the optimality of a steady-state strategy whenever a zero duality gap is obtained.
Let $\Phi$ be defined as 
$$\Phi = \{\varphi: \Delta_N^K \to \Delta_N^K \text{ injective and bounded}\}\enspace.$$
For a given function $\varphi \in \Phi$, we define the Lagrangian function $\mathcal{L}^{(\varphi)}: (\calA,\Delta_N^K,\bbR^{KN}) \to \bbR$ by 
$$\mathcal{L}^{(\varphi)}(a, \mu, \lambda):= r(a,\mu P(a)) + \left<\lambda,\varphi\left(\mu P(a)\right)-\varphi(\mu)\right>_{K N} \enspace.$$
As a direct consequence of the injectivity of $\varphi$, we obtain that for any given $\varphi\in\Phi$,$$\overline{g} = \max_{(a,\mu)\in\calA\times\Delta_N^K} \;\inf_{\lambda \in \bbR^{KN}} \mathcal{L}^{(\varphi)}(a,\mu,\lambda)\enspace.$$
We also define the dual problem $g^{(\varphi)}$ as 
\begin{equation}\label{eq::g_phi} 
g^{(\varphi)}:= \inf_{\lambda \in \bbR^{KN}} \; \max_{(a,\mu)\in\calA\times\Delta_N^K} \mathcal{L}^{(\varphi)}(a,\mu,\lambda)\enspace.
\end{equation}
The following result bounds the suboptimality gap
induced by the restriction to steady-state policies.
It will be applied in~\Cref{fig::lagrangian_bounds} to study the optimality
of such policies, in our application.
\begin{prop}\label{prop::dual_lagrangian}With $(g^*,h^*)$ solution of~\eqref{eq::ergodic_fixed_point} and $\overline{g}$ defined in~\eqref{eq::g_infty}, 
$$\overline{g} \leq g^*\leq g^{(\varphi)},\;\forall\varphi\in\Phi\enspace.$$
\end{prop}
\begin{proof}
The proof
%is given in Appendix~\ref{app::proof}, and
extends the arguments in~\cite[Remark 5.1]{Flynn_1979} to nonlinear functions $\varphi\in\Phi$.

First, from the geometrical convergence of the dynamics (see~\Cref{eq::contraction}), the valid strategy consisting in executing action $\overline{a}$ each period of time induces an average reward of $\overline{g}$, regardless the initial distribution. Therefore, $\overline{g}\leq g^*$.

Then, for $\epsilon > 0$, there exists $\lambda^\epsilon$ such that for any $(a,\mu) \in \calA\times\Delta_N^k$,
$$r(a,\mu P(a)) + \left<\lambda^\epsilon, \varphi(\mu P(a)) -\varphi(\mu)\right>_{K N}
\leq g^{(\varphi)} + \epsilon\enspace.$$
We construct a sequence of decision $a_1,\hdots,a_T$ leading to distribution $\mu_1,\hdots,\mu_T$. Then, at each period $t$,
$$r(a_t,\mu_t)+ \left<\lambda^\epsilon,\varphi(\mu_t) - \varphi(\mu_{t-1})\right>_{KN}\leq g^{(\varphi)} + \epsilon\enspace.$$
Therefore, we take the mean over $t=1,\hdots, T$ to recover the average reward criteria:
$$\frac{1}{T}\sum_{t=1}^T r(a_t,\mu_t)+ \frac{1}{T}\left<\lambda^\epsilon,\varphi(\mu_T) - \varphi(\mu_0)\right>_{KN}\leq g^{(\varphi)} + \epsilon\enspace.$$
The second term converges to zero when $T\to\infty$ as we suppose that $\varphi$ is bounded on the simplex. So,
$$\liminf_{T\to\infty} \frac{1}{T}\sum_{t=1}^T r(a_t,\mu_t) \leq g^{(\varphi)} + \epsilon \enspace. $$
The latter inequality is valid for any $\epsilon>0$, and any sequence of action $(a_t)_{t\in\bbN}$, so $g^* \leq g^{(\varphi)}$.
\end{proof}

We define the duality gap $\delta_{\mathcal{L}^{(\varphi)}}$ as
$$
\delta_{\mathcal{L}^{(\varphi)}}:=  g^{(\varphi)} - \overline{g}.$$ 
As an immediate consequence of~Proposition~\ref{prop::dual_lagrangian},
if there exists $\varphi\in\Phi$ such that $\delta_{\mathcal{L}^{(\varphi)}}=0$, then $g^* = \overline{g}$, and the dynamic program~\ref{eq::optimality_criteria} reduces to the static optimization program~\eqref{eq::g_infty}. Depending on the problem parameters, the duality gap may, or may not, vanish, see~\Cref{fig::lagrangian_bounds}.

%For a given function $\varphi$, a saddle point $(a^*,\mu^*,\lambda^*)$ of $\mathcal{L}^{(\varphi)}$ must satisfy
%\begin{subequations}\label{eq::saddle_point}
%\begin{align}
%&(a^*,\mu^*) \text{ maximizers of }\eqref{eq::g_infty}\\
%&\forall a,\mu,\,\left<\lambda^*,\varphi(\mu P(a)) \shortminus \varphi(\mu)\right> \leq r(a^*,\mu^*) \shortminus r(a,\mu P(a))
%\end{align}
%\end{subequations}
%\begin{prop}
%If $(a^*,\mu^*,\lambda^*)$ is a saddle point for $\mathcal{L}^{(\varphi)}$, then for any $d\in \bbR^{KN}$,
%$$\left<\lambda^*, D_d\varphi(\mu^*) (I - P(a^*))\right>_{KN} = r(a^*,dP(a^*))\,.$$
%\end{prop}

%%%%%%%%%%%%%%%%%%%%%%%%%%%%%%%%%%%%%%%%%%%%%%%%%%%%%%%%%%%%%%%%%%%%%%%%%%%%%%%%
\section{Application to electricity pricing}\label{sec::case_study}

\subsection{Description}
We suppose that an electricity provider has $N\shortminus 1$ different types of offers and that a study has distinguished beforehand $K$ customer segments, assuming that customers of a given segment have approximately the same behavior.
Given a segment $k$ and an offer $n\in [N\shortminus 1]$,
the \emph{reservation price} $R^{kn}$ is the maximum price that customers of
this segment are willing to spend on $n$, and $E^{kn}$ is the (fixed) quantity a customer of segment $k$ will purchase if he chooses $n$. The \emph{utility} for these customers is linear and is defined as $$U^{kn}(a):= R^{kn} - E^{kn}a^n \enspace.$$
where $a^n$ is the price for one unit of product $n$. The action space is then a compact subset of $\bbR^{N\shortminus 1}$.

%The model is said to be \emph{Unit-Demand} since each customer purchases exactly one contract, and \emph{Envy-free} since there is no limitation on the maximum number of customers able to purchase the same contract and so each customer chooses a contract maximizing his utility.

To model the competition between the provider and the other providers of the market, consumers have an alternative option (state of index $N$). We suppose that this alternative offer is fixed over time (for example a regulated contract). Then, under this assumption, it can be modelized w.l.o.g. by a null utility for each cluster ($U^{kN} = 0$).

If a customer of segment $k$ chooses the contract $n<N$ at price $a^n$, then the provider receives $E^{kn}a^n$ from the electricity consumption of the customer and has an induced cost of $C^{kn}$. Note that the cost should depend on the quantity $E^{kn}$, but as it is supposed to be a parameter, we omit this dependency. The (linear) reward for the provider is then
$$
\theta^{kn}(a) = E^{kn}a^n - C^{kn},\,n<N, \quad
\theta^{kN} = 0\enspace.
$$ 
We suppose that the transition probability follows a logit response, see e.g.~\cite{Pavlidis_2017}: 
\begin{equation}\label{eq::def_proba}
    [P^k(a)]_{n,m} = \frac{e^{\beta[U^{km}(a) + \gamma^{kn}\OneF_{m=n}]}}{\sum_{l\in[N]}e^{\beta[U^{kl}(a)+\gamma^{kn}\OneF_{l=n}]}}\enspace,
\end{equation}
where the parameter $\gamma^{kn}$ is the cost for segment $k$ to switch from contract $n$ to another one, and $\beta$ is the intensity of the choice (it can represent a ``rationality parameter''). 
One can easily check that~\ref{hypo::continuity_P}-\ref{hypo::bounded_r} are satisfied.

In the no-switching-cost case ($\gamma = 0$), we say that the customers response is \emph{instantaneous}, and corresponds to the classical logit distribution, see e.g.~\cite{Train_2009}:
\begin{equation}\label{eq::inst_response}
\mu^{kn}_L = e^{\beta U^{kn}(a)}\;/\;\sum_{l\in[N]}e^{\beta U^{kl}(a)}\enspace.
\end{equation}

% -----
\subsection{Steady-states}
The application scope of the transition model we defined in~\eqref{eq::def_proba} is broader than electricity pricing. For this specific kernel, we derive a closed-form expression for the stationary distributions, fully characterized by the instantaneous response:
\begin{theorem}\label{prop::proba_stat}
Given a constant action $a$, the distribution $\mu^k_t$ converges to $\overline{\mu}^k(a)$, defined as
\begin{equation}\label{eq::def_yinfty-old}
\overline{\mu}^{kn}(a) = \frac{\eta^{kn}(a) \mu^{kn}_L(a)}{\sum_{l\in [N]}\eta^{kl}(a) \mu^{kl}_L(a) }\enspace.
\end{equation}
where $\displaystyle \eta^{kn}(a) := 1+ \left[e^{\beta\gamma^{kn}}-1\right]\mu^{kn}_L(a)$, and $\mu_L$ is defined in~\eqref{eq::inst_response}.
\end{theorem}
\begin{proof}
In the proof, we forget the dependence on $k$ and $a$.
The stationary probability is defined as
$\forall m \in[N],\, \mu^m\left[1-P^{mm}\right] = \sum_{n\neq m}\mu^n P^{nm}\enspace.$
We can then replace by the definition of the probabilities~\eqref{eq::def_proba} to obtain
$$\mu^m\left[\frac{\sum_{l\neq m}e^{\beta U^m}}{\sum_l e^{\beta[U^l + \OneF_{l=m}\gamma^m]}}\right] = \sum_{n\neq m}\mu^n\left[\frac{e^{\beta U^m}}{\sum_l e^{\beta[U^l + \OneF_{l=n}\gamma^n]}}\right]\enspace.$$
Defining $\tilde{\mu}^n:= \frac{\mu^n}{\sum_l e^{\beta[U^l + \OneF_{l=n}\gamma^n]}},$ we obtain 
$$\forall m\in[N],\, \tilde{\mu}^m\sum_{l\neq m}e^{\beta U^l} = e^{\beta U^m}\sum_{l\neq m}\tilde{\mu}^l\enspace.$$
The solution $\tilde{\mu}^n := \lambda e^{\beta U^n}, n\in[N]$ is then a valid solution, and the constant $\lambda$ is chosen so that $\sum_{l\in[N]}\mu^l = 1$:
\begin{equation}\label{eq::def_yinfty}
\begin{aligned}
\overline{\mu}^{kn}(a) &= \lambda e^{\beta U^{kn}(a)} \sum_{m\in [N]}e^{\beta[U^{kn}(a) + \OneF_{m=n}\gamma^{kn}]}\\
\lambda^{-1} &= \sum_{n\in [N]}e^{\beta U^{kn}(a)} \sum_{m\in [N]}e^{\beta[U^{km}(a) + \OneF_{m=n}\gamma^{kn}]}
\end{aligned}
\end{equation}
Finally, $\eta^{kn} = \sum_l e^{\beta[U^{kl} + \OneF_{l=n}\gamma^{kn}]}/\sum_l e^{\beta U^{kl}}$. We recover the definition of $\overline{\mu}$~\eqref{eq::def_yinfty}.
\end{proof}
%The proof makes explicit the solution of $\mu^k P^k(a) = \mu^k$.
%, see~\ref{app::proof}.

As a consequence, the optimal steady-state can be found by solving
\begin{equation}\label{eq::optimal_steady_state}
\overline{g} = \max_{a\in\calA} r(a,\overline{\mu}(a))\enspace.
\end{equation}
Problem~\eqref{eq::optimal_steady_state} has no guarantee to be convex. However, it is a box-constrained smooth optimization problem which can be much more efficiently solved (at least up to local maximum) than the original time-dependent problem.  

In addition, if we suppose that $\gamma^{kn} = \gamma^k > 0$ for all $n$, then for any $a\in \calA$, we get the two following properties as immediate consequence of Theorem~\ref{prop::proba_stat}:
\begin{itemize}
\item $\displaystyle\lim_{\gamma^k\to 0} \overline{\mu}^{k}(a) = \mu^k_L(a)$,
\item $(\overline{\mu}^{kn})$ and $(\mu^{kn}_L)$ are sorted in the same order.
\end{itemize}
We now aim to compare the steady-state $\overline{\mu}^k(a)$ with the logit distribution $\mu_L^k(a)$ using the majorization theory:
\begin{defin}[Majorization,\cite{Marshall_2011}]For a vector $a\in\bbR^d$, we denote by $a^{\downarrow}\in \bbR^d$ the vector with the same components, but sorted in descending order. Given $a,b \in \Delta_d$, we say that $a$ majorizes $b$ from below written $a\succ b$ iff
$$\sum_{i=1}^k a^{\downarrow}_i \geq \sum_{i=1}^k b^{\downarrow}_i \quad \text{for }k=1,\hdots,d\enspace .$$
\end{defin}

\begin{prop}[Majorization property of the steady-state]\label{prop::majorization} Let $k\in[K]$ and $a\in \calA$ be given. Suppose that $\gamma^{kn} = \gamma^k > 0$ for all $n\in[N]$, then the stationary distribution majorizes the instantaneous logit response i.e.,
\begin{equation}\overline{\mu}^k(a) \succ \mu_L^k(a)\enspace .
\end{equation}
\end{prop}
\begin{proof}
    Let us suppose that we reorder the probabilities (and the $\eta$) such that they are sorted in the decreasing order.
$$\begin{aligned}
\left(\sum_{m=1}^n \overline{\mu}^m\right)^{\shortminus 1} &= \frac{\sum_{l=1}^n\eta^l \mu^l_L + \sum_{l=n+1}^N\eta^l\mu^l_L}{\sum_{m=1}^n\eta^m \mu^m_L}\\
&=1+\frac{\sum_{l=n+1}^N\eta^l \mu^l_L}{\sum_{m=1}^n\eta^m \mu^m_L}\\
&\leq 1+\frac{\sum_{l=n+1}^N \mu^l_L}{\sum_{m=1}^n \mu^m_L}
=\left(\sum_{m=1}^n \mu^m_L\right)^{\shortminus 1} \enspace .
\end{aligned}$$
The inequality comes from the sorting of $\eta$, and the last equality from $\sum \mu_L = 1$.
\end{proof}
Proposition~\ref{prop::majorization} establishes a qualitative feature of this model: if the price is kept constant over time, then, in the model with inertia, the stationary distribution of the population
{\em majorizes} the one obtained in the corresponding logit-model without inertia. Recalling that the majorization order expresses a form of dispersion, this means that inertia increases the concentration of the population on its favorite
offers.

\begin{lemma}\label{prop::maj_bound} Let us consider $a$ and $b$ in $\Delta_d$. If $a\succ b$, then for all $i$, $a_i \leq d b_i$.
\end{lemma}
\begin{proof}
$\displaystyle
a^{\downarrow}_i \leq \sum_{j=i}^d a^{\downarrow}_j = 1 - \sum_{j=1}^{i-1} a^{\downarrow}_j
\leq 1 - \sum_{j=1}^{i-1} b^{\downarrow}_j = \sum_{j=i}^d b^{\downarrow}_j 
\leq (d-i+1)b^{\downarrow}_i\leq d b^{\downarrow}_i \enspace .
$
\end{proof}
\begin{prop}[Boundedness of the steady-state gain]\label{prop::boundedness_gamma}
Even with $\mathcal{A} = \bbR^{N-1}$, the optimal steady-state gain $\overline{g}$ is bounded independently of $\gamma$.
\end{prop}
\begin{proof} Suppose that the optimal steady-state gain is attained for an action $a$, then
$$
\begin{aligned}
\overline{g}&=\sum_{k\in[K]}\rho_k\sum_{n\in [N]} (E^{kn} a^n - C^{kn})\overline{\mu}^{kn}(a)\\ 
&\leq \max_{k,n}(R^{kn}-C^{kn}) +\sum_{k\in[K]}\rho_k\sum_{n\in [N]} (E^{kn} a^n - R^{kn})\overline{\mu}^{kn}(a)  \\
&\leq \max_{k,n}(R^{kn}-C^{kn}) +\sum_{\substack{k\in[K]\\U^{kn}(a) < 0}}\rho_k \left<-U^{k}(a),\overline{\mu}^{k}(a)\right>_N \\
&\leq \max_{k,n}(R^{kn}-C^{kn}) + N \sum_{\substack{k\in[K]\\U^{kn}(a) < 0}}\rho_k\left<-U^{k}(a),\mu_L^k(a)\right>_N\\
&\leq \max_{k,n}(R^{kn}-C^{kn}) + \frac{N}{\beta e} \enspace.
\end{aligned}
$$
The third inequality comes from~\Cref{prop::maj_bound}. For the fourth one, since the logit expression contains a no-purchase option, $\mu_L^{kn}\leq \frac{1}{1+e^{-\beta U^{kn}(a)}}$. To conclude, it remains to see that $1+e^{\beta z} \geq (\beta e) z$ for all $z$.
\end{proof}
\Cref{prop::boundedness_gamma} proves that the optimal steady-state gain cannot diverge to infinity when the inertia grows. This qualitative result is no longer true for the optimal strategy (which may be a periodic sequence of actions instead of a single constant one), see~\Cref{sec::cycling_strategies}.

%%%%%%%%%%%%%%%%%%%%%%%%%%%%%%%%%%%%%%%%%%%%%%%%%%%%%%%%%%%%%%%%%%%%%%%%%%%%%%%%
\section{Numerical results}\label{sec::results}
The numerical results were obtained on a laptop i7-1065G7 CPU@1.30GHz. 
We solved the problem up to dimension 4 (2 provider offers, 2 clusters) with high precision ($\delta_\mu = 50$ points for each dimension, $1.6$ million discretization points, precision $\epsilon = 10^{\shortminus 5}$) in $7$ hours for RVI algorithm and in $70$ seconds for the Howard algorithm adapted to decomposable state-spaces (both methods parallelized on $8$ threads), see~\Cref{table::Howard}. The Policy Iteration algorithm adapted to decomposable state-spaces (\Cref{algo::Howard}) induces drastic computational time reductions with respect to the value-iteration algorithm and considerable memory gains in comparison with standard policy iteration procedure.

\begin{table}[!ht]
\centering
\begin{tabular}{|c|c||c|c|c|}
\hline
Instance & (node, arcs) & RVI & PI~\cite{Cochet_1998} & This work \\
    \hline\hline
    $K=1,N=1$ & \multirow{ 2}{*}{($2$e$3$,\,$2.5$e$6$)} & 70s & 1s & 0.2s\\
    $\delta_\mu = 1/2000$  & & 0.8Mo & 30Mo & 9Mo\\
    \hline
    $K=2,N=2$ & \multirow{ 2}{*}{($7.4$e$5$,\,$6.9$e$8$)} & 7h & 390s & 70s\\
    $\delta_\mu = 1/50$ & & 15Mo & 13Go & 103Mo\\
    \hline
\end{tabular}
\caption{Comparison RVI / Howard}\label{table::Howard}
\medskip\small
We provide running times that include the graph building step (which is a very costly operation for high dimensional graph in the standard PI algorithm. Each method ran on 8 threads.
\end{table}

In order to visualize qualitative results, we focus on the minimal non-trivial example (1 offer and 1 cluster). Note that the conclusions we draw from this example remain valid for the case $2$ offers / $2$ clusters. We use data of realistic orders of magnitude: we consider a population that checks monthly the market offers and consumes $E=500$kWh each month. The provider competes with a regulated offer of $0.17$\texteuro/kWh (inducing a reservation price of $85$\texteuro), and has a cost of $0.13$\texteuro/kWh. We suppose that the prices are freely chosen by the provider in the range $0.08$-$0.22$\texteuro/kWh. The intensity parameter $\beta$ is fixed to $0.1$.

\begin{figure*}[!ht]
\centering
\begin{subfigure}{.47\textwidth}
  \centering
  % include first image
  \includegraphics[width=\linewidth,clip=true,trim=.3cm 0cm 1.3cm 1.25cm]{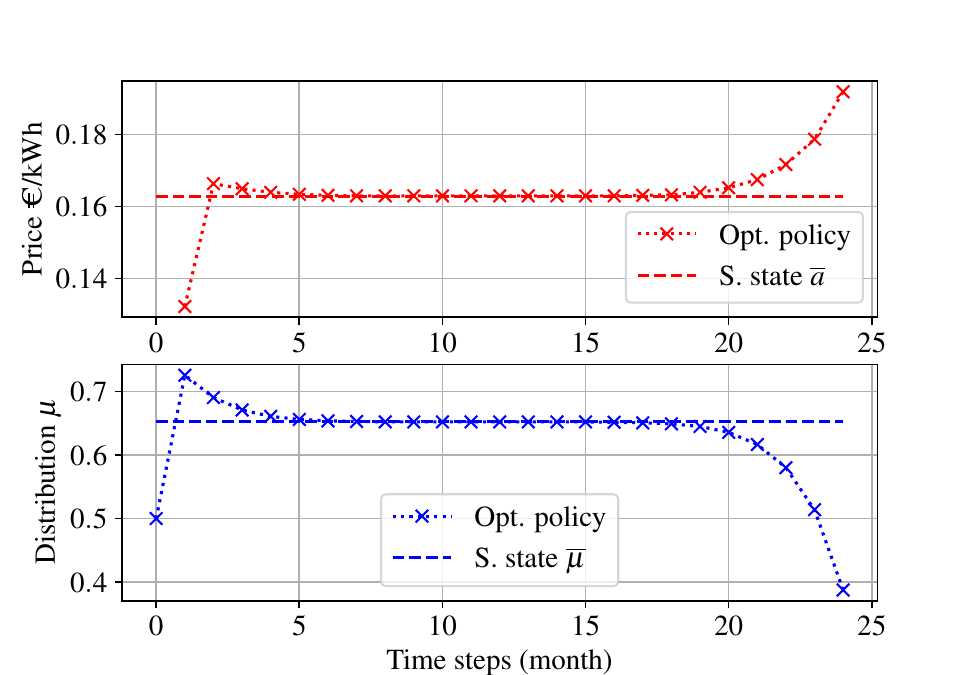}  
  \caption{Optimal finite horizon trajectory (provider action and customer distribution) for \emph{low} switching cost.}
  \label{fig::LFH_0-7}
\end{subfigure}
\hspace{0.015\linewidth}
\begin{subfigure}{.47\textwidth}
  \centering
  % include second image
  \includegraphics[width=\linewidth,clip=true, trim=.3cm 0cm 1.3cm 1.25cm]{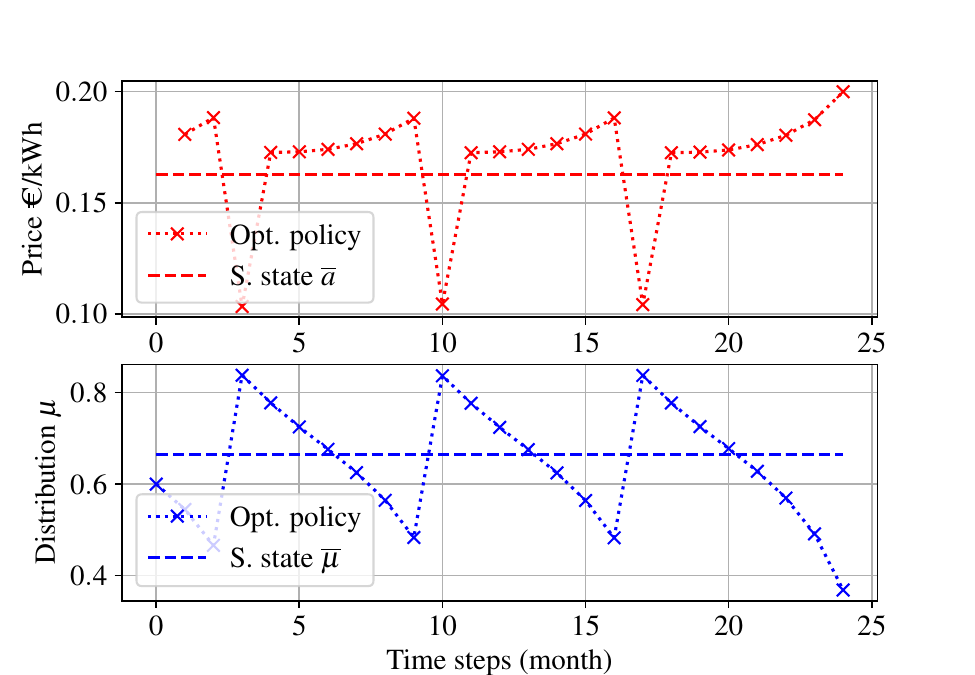}  
  \caption{Optimal finite horizon trajectory (provider action and customer distribution) for \emph{high} switching cost.}
  \label{fig::LFH_1-2}
\end{subfigure}
\\\vspace{.3\baselineskip}
\centering
\begin{subfigure}{.47\textwidth}
  \centering
  % include third image
  \includegraphics[width=\linewidth,clip=true,trim=.3cm 0cm 1.3cm 1.25cm]{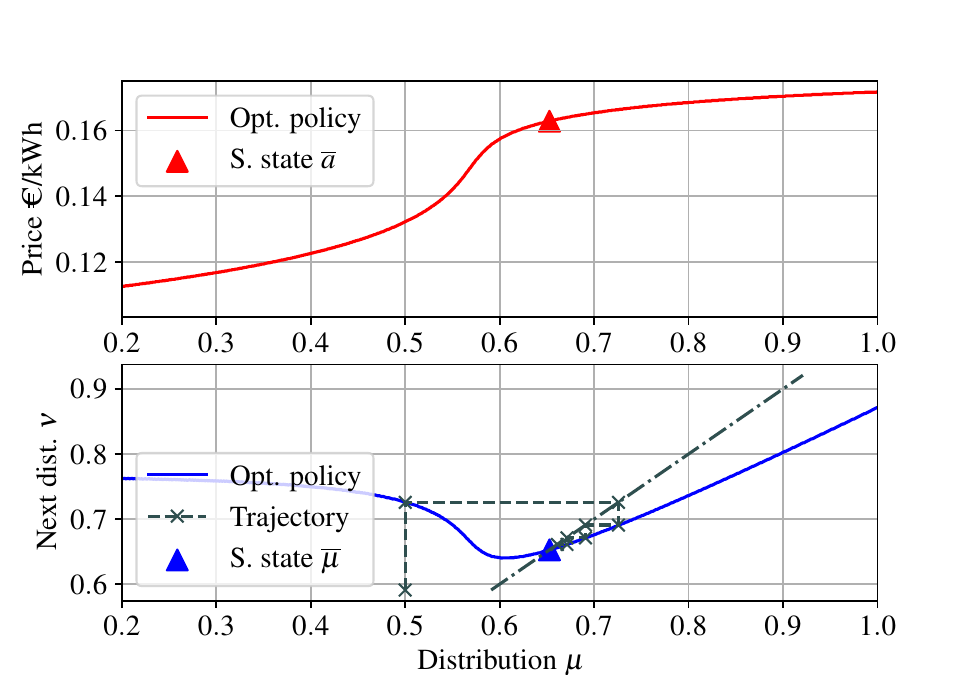}  
  \caption{Optimal decision for the long-run average reward (provider action and next customer distribution) for \emph{low} switching cost. Graphical iteration is drawn in dotted lines.}
  \label{fig::Erg_0-7}
\end{subfigure}
\hspace{0.015\linewidth}
\begin{subfigure}{.47\textwidth}
  \centering
  % include fourth image
  \includegraphics[width=\linewidth,clip=true,trim=.3cm 0cm 1.3cm 1.25cm]{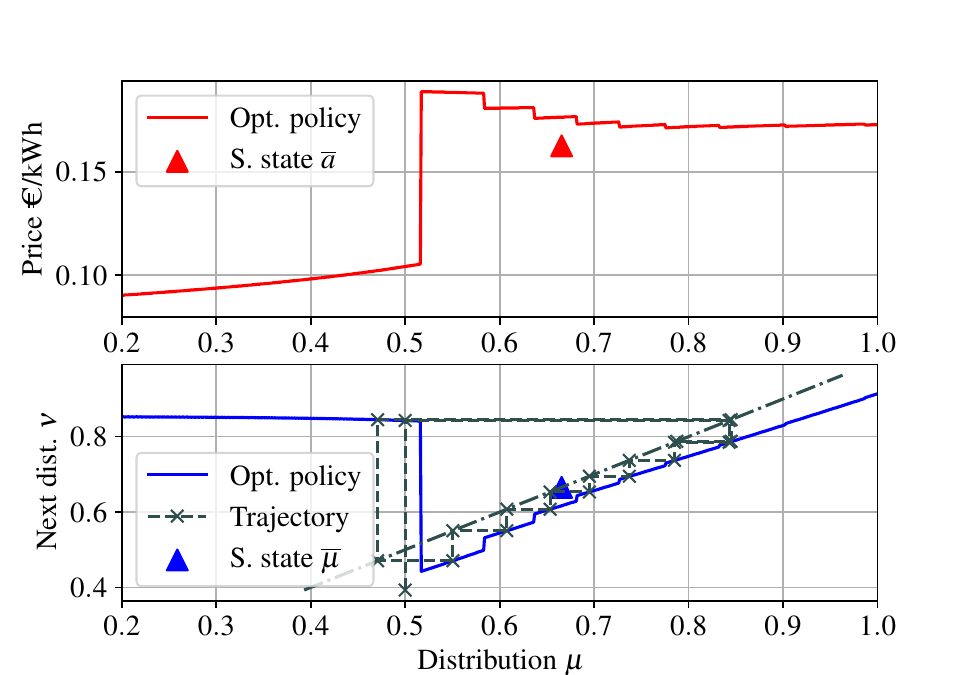}  
  \caption{Optimal decision for the long-run average reward (provider action and next customer distribution) for \emph{high} switching cost. Graphical iteration is drawn in dotted lines.}
  \label{fig::Erg_1-2}
\end{subfigure}
\caption{Numerical results for both the finite horizon and long-term average reward criteria.\\ \emph{Low} (resp. \emph{high}) switching cost stands for $\gamma = 20$ (resp. $25$). }
\label{fig::LFH_Erg_1D}
\end{figure*}

Numerical experiments in Fig.~\ref{fig::lagrangian_bounds}-\ref{fig::LFH_Erg_1D} emphasize the role of the switching cost. There exists a threshold -- around $\gamma=22$ in Fig.~\ref{fig::lagrangian_bounds} -- above which the steady-state policy become dominated by a cyclic strategy, where a period of promotion is periodically applied to recover a sufficient market share (period of $7$ time steps on this example, see Fig.~\ref{fig::LFH_1-2} and Fig.~\ref{fig::Erg_1-2}). Below this threshold, the optimal policy has an attractor point which is exactly the best steady-state price, see Fig.~\ref{fig::Erg_0-7}. The finite horizon policy is therefore a ``turnpike'' like strategy~\cite{Damm_2014}: we rapidly converge to the steady-state and diverge at the end of the horizon, see Fig.~\ref{fig::LFH_0-7}.  Fig.~\ref{fig::lagrangian_bounds} highlights that the adding of a convex function $\varphi$ strengthens the upper bound, so that the optimality of the steady-state strategy is guaranteed up to $\gamma$ around $19$.

\begin{figure}[!ht]
\centering
\includegraphics[width=0.9\linewidth, clip=true, trim=0.5cm 0.cm 1.5cm 1.2cm]{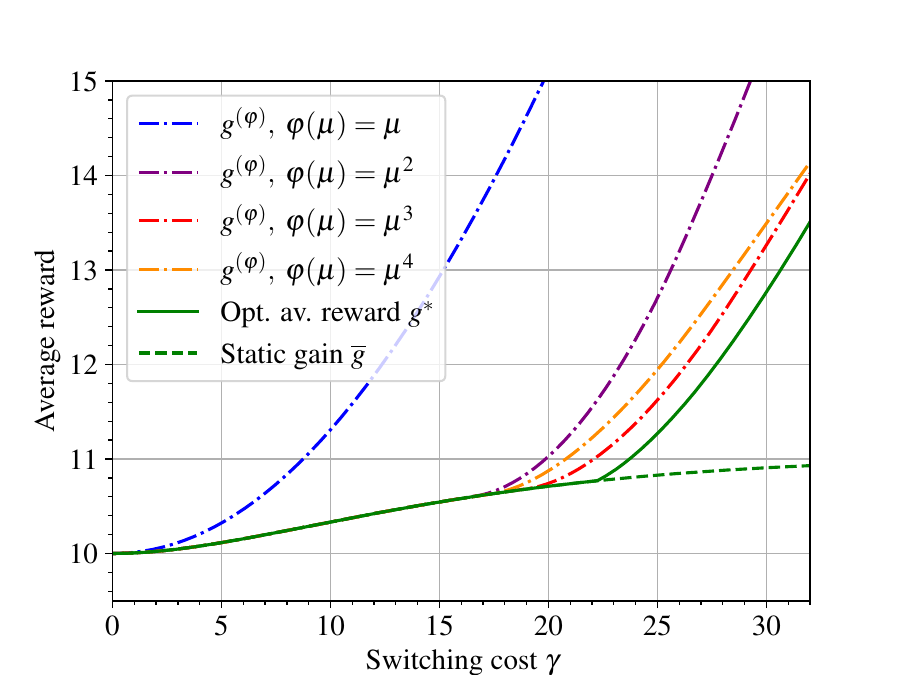}
\caption{Optimal gain $g^*$ for a range of switching costs, along with lower bound $\overline{g}$ and upper bounds $g^{(\varphi)}$, $\varphi(\cdot) = (\cdot)^{1,2,3,4}$.}
\label{fig::lagrangian_bounds}
\end{figure}

%%%%%%%%%%%%%%%%%%%%%%%%%%%%%%%%%%%%%%%%%%%%%%%%%%%%%%%%%%%%%%%%%%%%%%%%%%%%%%%%
\section{Study of the minimal non-trivial model}\label{sec::nontrivial_example}
Let us study the simple (yet non-trivial) case where the company has $1$ contract ($N=1$) and the population is homogeneous ($K=1$). Numerical results have been shown in previous section. 

In this setting, the probability $\mu$ to choose the retailer contract lies in the segment $[0,1]$. For the finite-horizon setting, the toy model is therefore defined as
\begin{equation}
\max_{a_1,\hdots,a_{T}\in \calA^T} \left\{\begin{split}
&\sum_{t=1}^{T} (a_t-C)\mu_{t}\\
&\st \, \begin{bmatrix}\mu_{t}& 1\shortminus\mu_{t}\end{bmatrix} = \begin{bmatrix}\mu_{t-1} & 1\shortminus\mu_{t-1}\end{bmatrix} P(a_t)
\end{split}
\right\}
\label{eq::toy_model}
\end{equation}
with $$ P(a_t) = 
\begin{bmatrix}
\frac{e^{\beta \gamma} e^{-\beta(a_t-R)}} {1+e^{\beta \gamma} e^{-\beta(a_t-R)}}
& \frac{1} {1+e^{\beta \gamma} e^{-\beta(a_t-R)}}\\
\frac{e^{-\beta(a_t-R)}} {e^{\beta \gamma} + e^{-\beta(a_t-R)}}
& \frac{e^{\beta \gamma}} {e^{\beta \gamma} +e^{-\beta(a_t-R)}}
\end{bmatrix}
$$
In the sequel, the data is $C = 2$, $R = 3$, $\beta = 3$, $T = 45$.

% ---------------------------------------------
\subsection{Cycling strategies}\label{sec::cycling_strategies}
\Cref{fig::LFH_Erg_1D} suggests a threshold (in terms of switching cost intensity) that separates the decision behavior into two different regimes : the convergence to a steady state for low switching costs intensity and the convergence to periodic strategies above the threshold (see~\Cref{fig::LFH_1-2}). Therefore, in order to better understand this cycling behavior, we define the set of periodic strategies in the one-dimensional case as follows:
\begin{defin} A \emph{$\tau$-cycle} is a cycling strategy of $\tau$ time steps, defined by
the customer response $(\mu_0,\hdots,\mu_{\tau-1},\mu_\tau)$, with the cycling condition $\mu_0=\mu_\tau$.
%The prices $(a_1,\hdots,a_\tau)$ are defined using~\eqref{eq::hat_x}.
For any $\tau$-cycle $l$, we denote by 
\begin{enumerate}[label=(\roman*)]
  \item $\overline{\mu}[l]=\frac{1}{\tau}\sum_{t=1}^\tau \mu_t$ 
and $V[l]=\frac{1}{\tau}\sum_{t=1}^\tau (\mu_t-\overline{\mu})^2$ the mean and the variance of the customer distribution over the cycle,
  \item$g[l]=\frac{1}{\tau}\sum_{t=1}^\tau (a_t-C)\mu_{t}$ the gain (mean profit over the cycle).
\end{enumerate}
\label{defin::tau_cycle}
\end{defin}

\begin{prop}
Let $\gamma > 0$, knowing $\mu_{t-1}$ and $\mu_t$ in $[0,1]$, there exists a unique $a_t$ verifying the
constraint in~\eqref{eq::toy_model}, defined as 
\begin{equation}\label{eq::hat_x}
\hat{a}_t :=e^{-\beta (a_t-R)} = \frac{ 2\mu_t-\kappa_t + \sqrt{ (2\mu_t-\kappa_t)^2 + 4\hat{\gamma}^2\mu_t(1-\mu_t) } }
                  { 2\hat{\gamma}(1-\mu_t) }
\end{equation}
where $\hat{\gamma} = e^{\beta \gamma}$ and $\kappa_t = 1 + (\hat{\gamma}^2 -1) (\mu_{t-1} - \mu_t)$.
\label{prop::hat_x}\end{prop}
\begin{proof}
From~\eqref{eq::toy_model}, one obtains the following equation:
$$
\mu_t = \left[\frac{\hat{\gamma}\hat{a_t}} {1+\hat{\gamma}\hat{a_t}} - \frac{\hat{a_t}}{\hat{\gamma}+\hat{a_t}}\right]\mu_{t-1} 
          + \frac{\hat{a_t}}{\hat{\gamma}+\hat{a_t}}\enspace,
$$
that can be equivalently written as a second-order equation:
$
0 = \hat{a}_t^2 \left[\hat{\gamma}(\mu_t-1)\right] + \hat{a}_t \left[2\mu_t - \kappa_t\right] 
  + \left[\hat{\gamma}\mu_t\right]
$
of discriminant $\Delta = (2\mu_t-\kappa_t)^2 + 4\hat{\gamma}^2\mu_t(1-\mu_t)\geq 0$.
\end{proof}
\begin{corol}\label{corol::hat_x_gamma_0}
As a special case of~\Cref{prop::hat_x}, 
\begin{enumerate}[label=(\roman*)]
\item if $\gamma = 0$, $\hat{a}_t = \frac{\mu_t}{1-\mu_t}$,
\item the steady-state policy that converges to $\mu\in ]0,1[$ is obtained by fixing the price to
\begin{equation}
\hat{a} = \frac{ 2\mu-1 + \sqrt{ (2\mu-1)^2 + 4\hat{\gamma}^2\mu(1-\mu) } }
                  { 2\hat{\gamma}(1-\mu) } \enspace .
\end{equation}
\end{enumerate}
\end{corol}
\begin{proof}
    Items (i) and (ii) are obtained with $\kappa_t = 1$, either with $\hat{\gamma} = 1$ or $\mu_{t-1}=\mu_t$.
\end{proof}

\Cref{prop::hat_x} gives an explicit expression of the (unique) action that allows a transition between state $\mu_{t-1}$ and $\mu_t$. The uniqueness can be extended to transitions $\mu_t = \mu_{t-1}P(a)$ in higher dimension, but the explicit characterization of the action is not straightforward.
We now want to compare the gain over a $\tau$-cycle $l$ and the steady-state gain $\overline{g}$. A first result is readily obtained in absence of switching costs, i.e., $\gamma = 0$, showing that constant-price policies are in this case optimal: 
\begin{prop}[Gain without switching cost]
Suppose that $\gamma=0$, then, the optimal steady-state policy induces a gain greater than the one achieved by any $\tau$-cycle $l$ of at least $\frac{V[l]}{\beta}$, i.e.,
$$g[l] \leq  \overline{g} - \frac{V[l]}{\beta} \enspace.$$
As a consequence, the optimal cycle corresponds to a constant-price policy.
\end{prop}
\begin{proof}
Using~\Cref{corol::hat_x_gamma_0},
$a_t = R - \frac{1}{\beta}\log\left(\frac{\mu_{t}}{1-\mu_{t}}\right),$
and the mean profit of a $\tau$-cycle $l$ is 
$$g[l] = (R-C) \overline{\mu}[l] - \frac{1}{\beta\tau}\sum_{t=1}^\tau \mu_t\log\left(\frac{\mu_t}{1-\mu_t}\right)\enspace.$$ 
The function $\mu \mapsto \mu\log\left(\frac{\mu}{1-\mu}\right)$ is strongly convex of modulus $1$. Therefore, using Jensen's inequality for strongly convex function, see e.g.~\cite{Merentes_2010}, we obtain that
$$g[l] \leq (R-C)\overline{\mu}[l] - \frac{1}{\beta}\overline{\mu}[l]\log\left(\frac{\overline{\mu}[l]}{1-\overline{\mu}[l]}\right) - \frac{V[l]}{\beta} \leq \overline{g} - \frac{V[l]}{\beta} \enspace.$$
\end{proof}
Let us specialize the $\tau$-cycles to a particular sub-class:
\begin{defin} A \emph{$(s,S,\tau)$-cycle} is a specific $\tau$-cycle, in which $\mu_t=S+\frac{s-S}{\tau}t, t\leq \tau$.
\end{defin}
\begin{prop}
Let us consider a $(s,S,\tau)$-cycle $l$. Then, 
$$g[l] \geq \frac{s(\tau-1)-S}{\tau}\gamma + O(1) \text{ as } \gamma\to\infty\enspace.$$
As a consequence, there exists a threshold $\Gamma> 0$ such that for any $\gamma\geq \Gamma$, the optimal steady-state policy is dominated by a $(s,S,\tau)$-cycle.
\end{prop}
\begin{proof}
Recalling that $\sqrt{a^2+b}\leq |a| +\frac{b}{2|a|}$, we have $\displaystyle\sqrt{(2\mu-\kappa)^2 + 4\hat{\gamma}^2\mu(1-\mu)}\leq \vert 2\mu-\kappa\vert + \frac{\hat{\gamma}^2\mu(1-\mu)}{\vert 2\mu-\kappa\vert}$.
We first look at a period $1\leq t < \tau$ where $\mu_{t-1} - \mu_t = \frac{S-s}{\tau}$.
As we suppose that $\gamma\to\infty$, $\hat{\gamma}\geq \sqrt{1 + \frac{\tau}{S-s}}$ and so $\kappa \geq 2\mu$. Therefore,
$$
\begin{aligned}
\hat{a} &\leq \frac{\hat{\gamma}\mu}{1+(\hat{\gamma}^2-1)\frac{S-s}{\tau}-2\mu}
\end{aligned}
$$
and 
$$
\begin{aligned}
a&\geq R+\frac{1}{\beta}\log\left(\frac{(\hat{\gamma}^2-1)(S-s) - \tau}{\hat{\gamma}\tau}\right)\\
&\simeq \frac{1}{\beta}\log(\hat{\gamma}) + O(1) = \gamma + O(1) \enspace.
\end{aligned}$$
If now we look at the last period $t=\tau$. Then, as we suppose that $\gamma\to\infty$, $\hat{\gamma}\geq \sqrt{1+\frac{1}{S-s}}$. Therefore,
$$
\hat{a} \leq \frac{1+(\hat{\gamma}^2-1)(S-s)}{\hat{\gamma}(1-y)}
+ \frac{\hat{\gamma}}{(\hat{\gamma}^2-1)(S-s)-1}
$$
and 
$$
\begin{aligned}
a&\geq R-\frac{1}{\beta}\log\left( \frac{1+(\hat{\gamma}^2\shortminus 1)(S-s)}{\hat{\gamma}(1-y)}
+ \frac{\hat{\gamma}}{(\hat{\gamma}^2 \shortminus 1)(S-s)-1}\right) \\
&\simeq -\gamma + O(1) \enspace.
\end{aligned}
$$
The mean profit is finally bounded by below :
$g[l] \geq \frac{1}{\tau}\left[(\tau-1)s-S\right]\gamma + O(1)$.
To conclude, any $(s,S,\tau)$-cycle satisfying $\tau \geq 1 + \frac{S}{s}$ induces a mean profit that diverges with respect to $\gamma$. In the meantime, the steady-state optimum is bounded, see~\Cref{prop::boundedness_gamma}, and so dominated for sufficiently large switching cost $\gamma$.

\end{proof}

\begin{figure}[!ht]
    \centering
     \includegraphics[width=\linewidth, clip=true, trim=0cm 0cm 0cm 0cm]{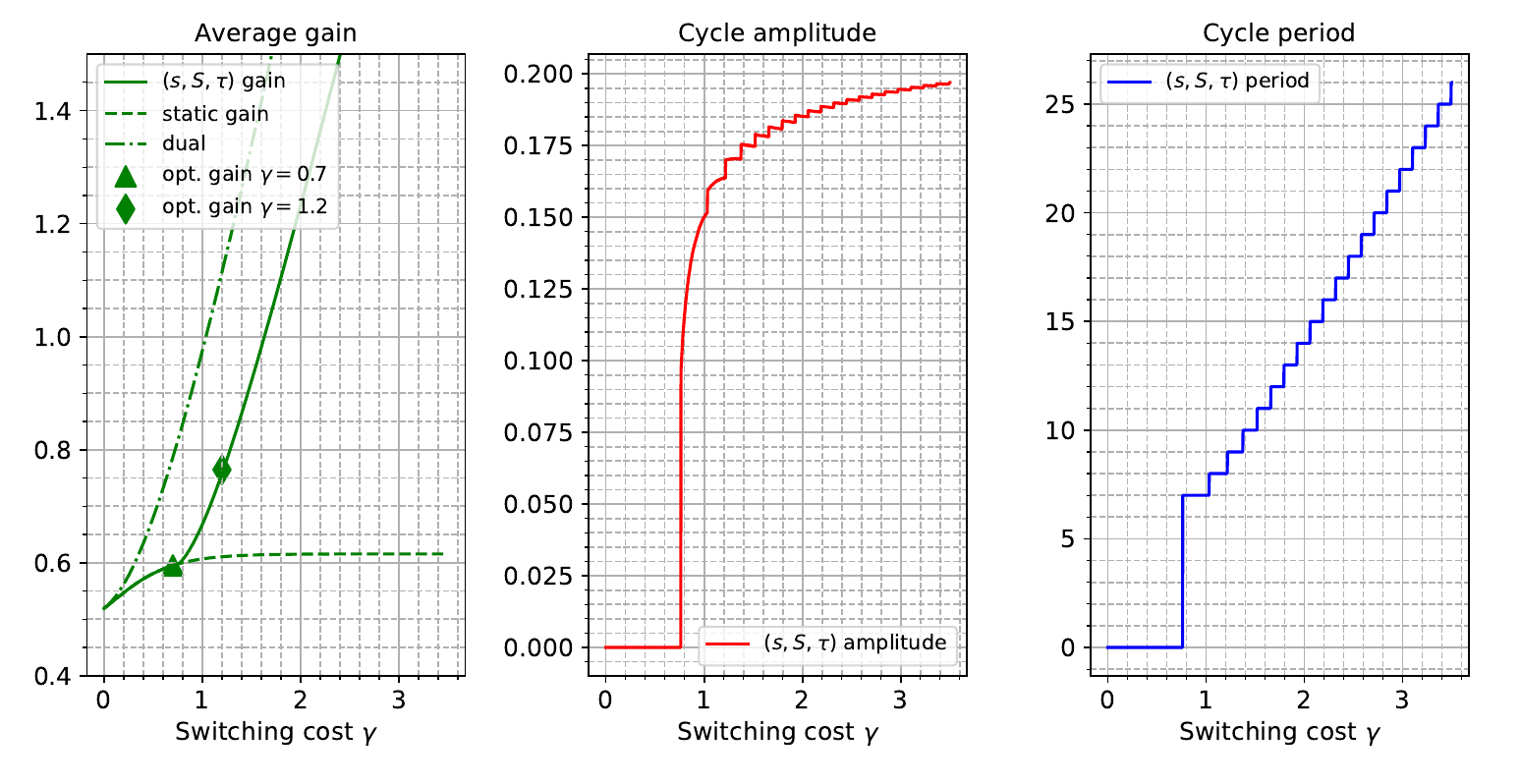}
     \caption{Evolution of the optimal $(s,S,\tau)$-cycle for a range of switching cost values}
  \label{fig::s_S_tau_cycle}
  \medskip\small
  The left (resp. middle, right) panel shows the gain (resp. cycle amplitude, cycle period) of the optimal $(s,S,\tau)$-cycle. The steady-state gain $\overline{g}$ is displayed for comparison, as well as the optimal gain obtained in~\Cref{fig::LFH_Erg_1D}. A kink appears at $\gamma\simeq 0.762$, indicating the separation of the cycling behavior from the steady-state behavior.
\end{figure}

In~\Cref{fig::s_S_tau_cycle}, we compute the optimal $(s,S,\tau)$-cycle, by iterating over the possible values of $s$, $S$, and $\tau$ for each given value of $\gamma$. Before the kink, the optimal cycle is in reality the constant-price strategy (cycle of amplitude $0$), and after this point, there exists cycle of positive amplitude that outperforms the steady-state strategy. The results found in~\Cref{fig::LFH_Erg_1D} for a broader class of cycles are consistent, and the $(s,S,\tau)$-cycles are good approximations of the optimal policy.

% ======================================
\section{Conclusion}
We developed an ergodic control model to represent the evolution of a large population of customers, able to  actualize their choices at any time. Using qualitative properties of the population dynamics (contraction in Hilbert's projective metric), we showed the existence of a solution to the ergodic eigenproblem (in the presence of noise or in the deterministic setting), which we applied to a problem of electricity pricing. A numerical study reveals the existence of optimal cyclic promotion mechanisms, that have already been observed in economics, and we proved this behavior on a toy example. We also quantified the sub-optimality of constant-price strategy in terms of a specific duality gap.

The present model has connections with partially observable MDPs, in which the state space is also a simplex. We plan to explore such connections in future work. We also aim at analyzing the problem through the weak KAM angle. In particular, we could expect to obtain a turnpike-like property when the Aubry set (to which the dynamics converge under any optimal policy) is reduced to a singleton. Finally, the approximation error induced by the discretization should be explored. In particular, exploiting the contraction of the dynamics may allow us to obtain convergence ratios, using similar arguments as in~\cite{Bayraktar_2023}.

%==========================
\printbibliography[heading=bibintoc]

\end{document}